\documentclass[a4paper,11pt,leqno]{article}	
\usepackage{graphicx}
\usepackage{relsize}
\usepackage{hyperref}

\usepackage{comment}

\medskipamount

\usepackage{amsmath}
\usepackage{amsthm}
\usepackage{amssymb}
\usepackage{amscd}
\usepackage{enumerate}
\usepackage{pstricks}
\usepackage{pst-node}
\usepackage{esint}
\usepackage{mathtools}

\theoremstyle{plain}
\newtheorem{theorem}{Theorem}[section]

\newtheorem{lemma}[theorem]{Lemma}
\newtheorem{proposition}[theorem]{Proposition}

\theoremstyle{definition}
\newtheorem{assumption}[theorem]{Assumption}
\newtheorem{definition}[theorem]{Definition}

\theoremstyle{remark}
\newtheorem{remark}{Remark}

\def\bs{}

\newcommand{\p}{\varphi}
\renewcommand{\d}{\partial}
\newcommand{\dd}{\partial}

\renewcommand{\O}{\Omega}

\newcommand\I{I}
\newcommand\QQ{\mathbb{S}}
\newcommand\B{B}
\newcommand\Q{Q}

\newcommand\V{V}
\newcommand\A{A}
\newcommand\Psib{\Psi}
\newcommand\R{\mathbb{R}}
\newcommand\M{\mathbb{M}}
\newcommand\N{\mathbb{N}}
\newcommand\Z{\mathbb{Z}}

\newcommand\taub{\tau}
\newcommand\etab{\eta}

\newcommand\Ma{M}

\newcommand\Ra{R}

\renewcommand\AA{{\mathbf S}}
\newcommand{\HH}{\mathcal{H}}

\newcommand\zetaa{{\boldsymbol \zeta}}
\newcommand\otto{\xrightharpoonup{osc,\gamma}}
\newcommand\stto{\xrightarrow{2,\gamma}}
\newcommand\hx{x}
\renewcommand\boldsymbol{}

\newcommand{\bbv}{b_V} 
\newcommand\ya{y}
\newcommand\ua{u}
\newcommand\phia{\phi} 
\newcommand\ra{r}
\newcommand\wa{w}
\newcommand\na{n}

\newcommand\ca{c}

\newcommand\va{v}

\newcommand\pa{p}
\newcommand\weak{\rightharpoonup}
\renewcommand{\t}{\widetilde}
\renewcommand{\o}{\tilde}

\newcommand\calB{\mathcal B}

\newcommand\calW{\mathcal{W}}

\newcommand\calY{\mathcal{Y}}

\newcommand\ttan{\textrm{tan}}

\newcommand{\SO}[1]{\operatorname{SO}(#1)}
\newcommand\id{I}
\newcommand\sym{\operatorname{sym}}

\newcommand{\so}[1]{\operatorname{so}(#1)}

\newcommand\wto{\rightharpoonup}

\newcommand\wtto{\xrightharpoonup{2,\gamma}}

\newcommand\esssup{\mathop{\operatorname{ess\,sup}}}
\newcommand\ud{\,\mathrm{d}}
 
 \newcommand\dist{\operatorname{dist}}
\newcommand{\e}{\varepsilon}
\newcommand{\eh}{{\varepsilon{\scriptstyle(h)}}}

\DeclareMathOperator{\curl}{curl}

\DeclareMathOperator{\cof}{cof}

\newcommand\ee{\mathbf e}

\title{Derivation of a homogenized von-K\'arm\'an shell theory from 3D elasticity}
\date{\today}

\author{Peter Hornung\\[1ex]
Max Planck Institute for Mathematics in the Sciences\\
  Inselstr. 22 \\ D-04103 Leipzig\\Germany\\
email: hornung@mis.mpg.de\\[2ex]
Igor Vel\v{c}i\'c\\[1ex]
Max Planck Institute for Mathematics in the Sciences\\
  Inselstr. 22 \\ D-04103 Leipzig\\Germany\\
\newline \newline
  BCAM-Basque Center for Applied Mathematics\\ Alameda de Mazarredo 14\\
48009 Bilbao\\ Bizkaia (Basque-Country, Spain)\\
email: velcic@mis.mpg.de}

\begin{document}

\maketitle

\begin{abstract}
We derive the model of homogenized von K\'arm\'an shell theory,
starting from three dimensional nonlinear elasticity.
The original three dimensional model contains two small parameters:
the oscillations of the material $\e$ and the thickness of the shell $h$.
Depending on the asymptotic ratio of these two parameters,
we obtain different asymptotic theories.
In the case $h\ll\e$ we identify two different asymptotic theories,
depending on the ratio of $h$ and $\e^2$.
In the case of convex shells we obtain a complete picture in the whole
regime $h\ll\e$.

 \vspace{10pt}


\end{abstract}

\tableofcontents[1]

\section{Introduction} \label{prvi}

This paper is about von-K\'arm\'an theory for thin elastic shells.
There is a vast literature on shell theory in elasticity. An
overview about the derivation of models for linear and nonlinear shells by the method of formal
asymptotic expansions can be found in \cite{Ciarletshell00}.
In the case of linearly elastic shells, the models thus obtained can also be justified by a rigorous convergence result,
starting from three dimensional linearized elasticity (see also \cite{Ciarlet-Lods-96,Ciarlet-Lods-Miara-96}).
\\
In the last two decades, rigorous justifications of nonlinear models for
rods, curved rods, plates and shells were obtained by means of $\Gamma$-convergence,
starting from three dimensional nonlinear
elasticity. The first papers in that direction are
\cite{AcBuPe91,LDRa95,LDRa96} for the string model, membrane plate and
shell model respectively. The rigorous derivation of
nonlinear bending theory of plate was achieved in \cite{FJM-02}; see also \cite{Pa01} for
an earlier result in this direction. F\"oppl-von K\'arm\'an theory for plates was derived in \cite{FJM-06}.
In \cite{MoMu03,MoMu04}, bending and von K\'arm\'an theories for rods were derived.
In \cite{FJMMshells03} the nonlinear bending theory shell model was derived, and
in \cite{Lewicka1} the von K\'arm\'an shell model was derived.

Here we are interested in an the ansatz-free derivation of a homogenized von K\'arm\'an
shell theory by simultaneous homogenization and dimension reduction.
Our starting point is the energy functional from $3d$ nonlinear elasticity. It attributes to a
deformation $u$ of a given shell $S^h \subset \R^3$ of small thickness $h>0$ around a surface $S \subset \R^3$
the stored elastic energy
\begin{equation}\label{intro:1}
  \frac{1}{h^4\,|S^h|}\int_{S^h}W_\e(x,\nabla \bs u(x))\,dx,\qquad
  \bs u\in H^1(S^h,\R^3).
\end{equation}
Here $W_\e$ is a non-degenerate stored energy function that oscillates periodically
on the surface, with some period $\e\ll 1$. We are interested in the
effective behavior when both the thickness $h$ and the period $\e$ are
small. The separate limits $h\to 0$ and $\e\to 0$ are reasonably well understood:
In \cite{Lewicka1} it is shown that, when $W_{\e}$ does not depend on $\e$, then
the functionals \eqref{intro:1}  $\Gamma$-converge as
$h\to 0$ to a two-dimensional
von K\'arm\'an shell theory. Regarding the limit $\e\to
0$, which is related to homogenization, the first rigorous results relevant in nonlinear
elasticity were obtained by Braides~\cite{Braides-85} and
independently by M\"uller~\cite{Mueller-87}. They proved that, under suitable growth
assumptions on $W_\e$, the energy \eqref{intro:1} $\Gamma$-converges as $\e\to 0$ (and $h$
fixed) to the functional obtained by replacing $W_\e$ in
\eqref{intro:1} with the homogenized energy
density given by an infinite-cell homogenization formula.

In this paper we study the asymptotic behavior when both the
thickness $h$ and the period $\e$ tend to zero \textit{simultaneously}. As
a $\Gamma$-limit we obtain a two-dimensional von K\'arm\'an shell model with
homogenized material properties.
Recently, the von K\'arman plate model (see \cite{NeuVel-12}), the bending
plate models (see \cite{Horneuvel12,Vel13}), and bending rod models (see \cite{Neukamm-10,Neukamm-11}),
were analyzed in this way. Simultaneous homogenization and dimensional reduction was also done in
the case of periodically wrinkled plate (see \cite{Velcic-12}).
As explained there, in these cases one does not obtain infinite-cell homogenization formula like in the membrane case
(see \cite{Braides-Fonseca-Francfort-00,Babadjian-Baia-06}). The basic reason for that is the
fact since we are in small strain regimes, the energy is essentially convex in the strain.
This is the main reason why we can use two scale convergence techniques in all these cases.
However, every case has its own peculiarities. In the von K\'arm\'an theory of plates, one obtains
a limiting quadratic energy density which is continuous in the asymptotic ratio $\gamma$ between $h$ and $\e$, for all $\gamma \in [0,\infty]$.
Moreover, the case $\gamma=0$ corresponds to the situation when the dimensional reduction dominates and the
obtained model is just the homogenized von K\'arm\'an plate model.
The situation $\gamma=\infty$ corresponds to the case when homogenization dominates and the
obtained model is the von K\'arm\'an plate model of the homogenized functional.
The case of bending plate is more involving;
we are able to obtain the models in the case $\gamma \in (0,\infty]$ (see \cite{Horneuvel12}) and in the case
$\gamma=0$ under the additional assumption that $\e^2 \ll h\ll \e$ (see \cite{Vel13}). This model does not correspond to the situation of the homogenized bending plate model, but is the limiting situation of the models when $\gamma \to 0$ and $\gamma>0$.

In case of von K\'arm\'an shell theory studied in the present paper,
we encounter two different scenarios in the regime $h\ll\e$, depending whether $h\sim\e^2$ or $h\ll\e^2$.
Our main result is presented in Theorem \ref{tm:glavni}.
We are not able to cover the case
$h\ll\e^2$ in a generic way for arbitrary reference surfaces $S$. A stronger influence of the geometry of the reference surface $S$ is expected in this case.
In fact, in the case when $S$ is a convex surface, we succesfully derive the limiting model even for the regime $h\ll\e^2$, see Theorem \ref{tm:glavni2}.

Our analysis requires both techniques from {dimension reduction}, in particular, the quantitative rigidity estimate and approximation schemes developed in
\cite{FJM-02,FJM-06}; and techniques from {homogenization methods}, in particular, two-scale
convergence \cite{Nguetseng-89, Allaire-92,Visintin-06,Visintin-07}. To our knowledge our result is the first rigorous result combining homogenization
and dimension reduction for shells in the von K\'arm\'an regime. The homogenization for linearly elastic shells was carried out in \cite{Lutoborski-85}.

This paper is organised as follows: after introducing the setting and basic objects in Section \ref{drugi} and \ref{treci} we state the main result in Section \ref{treci}. In Section \ref{Compactness} we identify the two scale limit of the strain and prove lower bound for $\Gamma$-limit. In Section \ref{peti} we construct the recovery sequences and thus prove the upper bound.
All these results are given for general surfaces and the cases $h \gg \e^2$ or $h \sim \e^2$.
In the last section we analyze the case of convex shells for the situation when $h\ll \e^2$.

\subsection*{Notation}

The notation $A\lesssim B$ means that $A\leq CB$ with $C$ depending only on quantities regarded as constant
in the context in question.

In this paper we
frequently encounter function spaces of periodic functions. We
denote by $\calY$ the real line $\R$ equipped with the torus topology, that is
$y{+}1$ and $y$ are identified in $\calY$. We write $C(\calY)$
to denote the space of continuous functions $f:\,\R\to\R$
satisfying $f(y+1)=f(y)$ for all $y\in\R$. Clearly, $C(\calY)$
endowed with the norm $||f||_\infty:=\sup_{y\in Y}|f(y)|$ is a
Banach space. Moreover, we set $C^k(\calY):=C^k(\R)\cap
C(\calY)$ and denote by $L^2(\calY)$, $H^1(\calY)$ and
$H^1(S{\times}\calY)$ the closure of $C^\infty(\calY)$ and
$C^\infty(\bar S;C^\infty(\calY))$ w.~r.~t. the norm in
$L^2(Y)$, $H^1(Y)$ and $H^1(S{\times} Y)$, respectively. By
$\dot{L}^2 (\calY)$, $\dot{H}^k(\calY)$ we denote the subspace of functions
$H^k(\calY)$ whose mid-value over $\calY$ is zero. Obviously,
all these spaces are Banach spaces.
For $A\subset\R^d$ measurable and $X$ a Banach space, $L^2(A;X)$ is understood in the
sense of Bochner. We tacitly identify the spaces $L^2(A;L^2(B))$ and
$L^2(A\times B)$; since whenever $f\in L^2(A\times B)$, then there
exists a function $\tilde f\in L^2(A;L^2(B))$ with $f=\tilde f$
almost everywhere in $A\times B$.
By $(e_1,e_2,e_3)$ we denote the standard basis on $\R^3$.

\section{Geometric preliminaries and general framework} \label{drugi}

In this subsection we do not always display the explicit regularity assumptions; the minimal requirements are obvious.
We assume that $\omega\subset\R^2$ is a bounded domain with boundary of class $C^3$.
We set $I:=(-\tfrac{1}{2},\tfrac{1}{2})$ and $\Omega^h:=\omega\times(hI)$, and $\Omega:=\omega\times I$.
The variables on $\omega$ (resp. $\Omega$) will be denoted by
$\xi_1,\xi_2$ (resp. $\xi_1, \xi_2, t$).
For a function $f:\Omega \to \R^3$ we define $\nabla_h f:=(\partial_1 f\, , \partial_2 f\, ,\,\tfrac{1}{h} \partial_3 f)$.

Let $S$ be a compact connected oriented surface with boundary which is embedded in $\R^3$.
For convenience we assume that $S$ is parametrized by a single chart: From now on,
$\psi\in C^3(\overline{\omega}; \R^3)$ denotes an embedding 
with $\psi(\omega) = S$. The inverse of $\psi$ is denoted by $\ra: S \to \omega$, and we assume it to be of class $C^3$.
We leave it to the interested reader to verify to which extent these regularity assumptions on $S$ can be weakened without altering our arguments.
\\
The nearest point retraction of a tubular neighbourhood of $S$ onto $S$ will be denoted by $\pi$.
Hence
$$
\pi (\hx + t\na(\hx)) = \hx \quad  \mbox{ whenever $|t|$ is small enough.}
$$
We introduce the basis vectors of the tangent bundle determined by $\psi$, namely the push-forwards
$
\taub_i = \psi_*e_i
$. Explicitly, this means
$$
\taub_i(x) = (\partial_i \psi)(r(x)) \mbox{ for }i = 1, 2\mbox{ and all }x\in S.
$$
By our hypotheses on $\psi$ there exist $\eta_1, \eta_2>0$ such that
\begin{eqnarray} \label{eq:0}
\eta_1 \leq \det ([\taub_1\ \taub_2]^T
[\taub_1\ \taub_2]) \leq \eta_2,\
\|\taub_1\|_{W^{2, \infty}(S)} \leq \eta_2, \|\taub_2\|_{W^{2, \infty}(S)} \leq \eta_2.
\end{eqnarray}
We denote by $(\taub^1(\hx),\taub^2(\hx))$
the dual base to the base
$(\taub_1(\hx),\taub_2(\hx))$, that is,
$$
\tau^i(\hx) = (\dd^i\psi)(\ra(\hx)).
$$
By $\na : S\to\mathbb{S}^2$ we denote the unit
normal, that is,
$$
n(x) = \frac{\tau_1(x)\wedge\tau_2(x)}{|\tau_1(x)\wedge\tau_2(x)|}\mbox{ for all }x\in S.
$$
By $T_{\hx} S = \mbox{span }\{\tau_1(x), \tau_2(x)\}$
we denote the tangent space to $S$ at $\hx$
For each
$\hx \in S$, the vectors $\taub_1(\hx),\taub_2(\hx)$ and $\na(\hx)$ form a
basis of $\R^3$. Its dual basis is $(\taub^1(\hx),\taub^2(\hx),
\na(\hx))$. We define $\taub_3(\hx)=\taub^3(\hx)=\na(\hx)$.

For a subset $A \subset S$ we set
$$
A^h=\{\hx + t\na (\hx); \ \hx \in S,-h/2<t<h/2 \}.
$$
In particular, the shell is given by
$$
S^h=\{\hx+t\na (\hx); \ \hx \in S, -h/2<t<h/2 \}.
$$
We introduce the function $t : S^1\to\R$ by
\begin{equation}
\label{deft}
t(x) = (x - \pi(x))\cdot n(x)\mbox{ for all }x\in S^1.
\end{equation}
By $\ra_e:S^{1}  \to
	\Omega$ we denote the map  (see below why we assume that $\pi$ and $t$ is well-defined on $S^1$)
	$$\ra_e(x)=\ra(\pi(x))+t(x) e_3. $$
Clearly,
	\begin{equation}
	\ra_e^{-1}(\xi_1,\xi_2,t)= \psi(\xi_1,\xi_2) + t \na \big(\psi(\xi_1,\xi_2)\big),
	\end{equation}
	and thus:
	\begin{eqnarray}
	\label{eq:400}\nabla \ra_e^{-1} (\xi_1,\xi_2,t)&=&(\I+t\AA(x))[\taub_1(x),\taub_2(x),\taub_3(x)], \quad \textrm{where } x=\psi(\xi_1,\xi_2) \\
	\label{eq:500}\nabla \ra_e (x)&=&[\taub^1(x),\taub^2(x),\taub^3(x)]^T (\I+t(x)\AA(x))^{-1}.
	\end{eqnarray}

We denote by
$$
T_S(x) := I - n(x)\otimes n(x)
$$
the orthogonal projection from $\R^3$ onto $T_x S$.
We will frequently deal with  vector fields $V : S\to\R^3$ on the surface.
We extend all such vector fields trivially from $S$ to $S^1$, simply by
defining $V(x) = V(\pi(x))$ for all $x\in S^1$.
By $V_{\ttan}$ we denote the projection of vector field $V$ on the tangential space i.e. $V_{\ttan}=T_S V$.
We will denote by $\o{V}$ the corresponding vector field along $\omega$, i.e. we set
$
\o{V}(x) = V(\psi(\xi))\mbox{ for all } \xi\in\omega.
$

The space of quadratic forms on $S$ is denoted by $\QQ$. It consists of all maps $B$ on $S$ such that, for
each $x\in S$, the map
$$
B(x) : T_xS\times T_xS\to\R
$$
is symmetric and bilinear. We will frequently
regard $B$ as a map from $S$ into $\R^{3\times 3}$ via the embedding $\iota$ defined by
$$
\iota(B) = B(T_S, T_S).
$$
On the right-hand side and elsewhere we identify bilinear maps from $\R^3$ into itself with $\R^{3\times 3}$.
By definition, $B(T_S, T_S) : S\to\R^{3\times 3}$ takes the vector fields
$v, w : S\to\R^3$ into the function
$$
B(T_Sv, T_Sw).
$$
By definition, $B\in L^2(S; \QQ)$ means that (using the above embedding) $B\in L^2(S; \R^{3\times 3})$ and $B\in\QQ$.
The spaces $H^1(S; \QQ)$ etc. are defined similarly.
By $\QQ(x)$ we denote the set of all quadratic forms on $T_xS$ which can be embedded in the space $\R^{3\times3}$. By $\QQ(x)_{\sym}$ we denote the set of symmetric quadratic forms on $T_x S$ which can be embedded in $\R^{3 \times 3}_{\sym}$.
\\
For a function $f : S\to\R^3$ we regard its tangential derivative $\nabla_{tan} f(x)$ as a linear map from
$T_x S$ into $\R$. For a tangent vector field $\tau$ along $S$ we write $\dd_{\tau} f = \nabla_{tan} f\ \tau$.
A similar notation applies to vector fields instead of functions.
By $\nabla_{\ttan} \nabla_{\ttan} f$ we denote the triilinear  form $\nabla_{\ttan}{\nabla_{\ttan}} f (\eta_1, \eta_2,\eta_3)= \partial_{\eta_2} \partial_{\eta_3} f \cdot \eta_1$. For scalar $f$, $\nabla_{\ttan} \nabla_{\ttan} f$ is just the bilinear form $\nabla_{\ttan} \nabla_{\ttan} f (\eta_1, \eta_2)=\partial_{\eta_1} \partial_{\eta_2} f$

The Weingarten map $\AA$ on the surface $S$
is given by $\AA = \nabla_{tan} n$, i.e.,
$$
\AA(x)\tau = (\dd_{\tau} n)(x) \mbox{ for all }x\in S,\ \tau\in T_x M.
$$
We extend $\AA$ to a linear map on $\R^3$ by setting $\AA = \AA\ T_S$, i.e., we define $\AA(x)n(x) = 0$.
Moreover, we extend $\AA$ trivially from $S$ to $S^1$, i.e., we have $\AA(x) = \AA(\pi(x))$.
With a slight abuse of notation, we denote by $\AA$ also the (negated) second fundamental form of $S$ defined by
$$
\AA_{ij}(\hx) :=  \AA (\hx) \taub_i(\hx) \cdot\taub_j(\hx). 
$$
In general, for a given bilinear form $\B$ on $S$ we denote its local coordinates by
$$
\B_{ij} := \B \taub_j \cdot \taub_i.
$$
Obviously $\B = \sum_{i,j=1}^2 \B_{ij} \taub^i \otimes \taub^j$.
\\
After rescaling the ambient space, we may assume that the curvature of $S$ is as small as we please.
In particular, we may assume without loss of generality that $\pi$ is
well-defined on a domain containing the closure of $\{\hx+t\na (\hx); \ \hx \in S, -1 < t < 1\}$,
and that
$$
1/2 < |Id + t\AA(\hx)| < 3/2
$$
for all $t\in (-1, 1)$ and all $\hx\in S$.

\begin{lemma}\label{dpi}
For all $x\in S^1$
we have

$$
(\nabla \pi)(x) = T_S(\pi(x)) \left( I + t(x) \AA(\pi(x))\right)^{-1}.
$$
\end{lemma}
\begin{proof}
Let $x\in S$, let $\tau\in T_x S$
and let $\gamma\in C^1((-1, 1), M)$ with $\gamma(0) = x$ and $\dot{\gamma}(0) = \tau$. Then
$$
\pi(\gamma + sn(\gamma)) = \gamma \mbox{ on }(-1, 1).
$$
Taking the derivative with respect to the arclength of $\gamma$, this implies
$$
(\nabla \pi)(\gamma + tn(\gamma))(\tau + t\AA(\gamma)\tau) = \tau.
$$
As $x\in S$ and $\tau\in T_x S$ were arbitrary, we conclude that
\begin{equation}
\label{dpi-1}
 (\nabla \pi)(x + tn(x))(I + t\AA(x)) = T_S(x)
\end{equation}
on $T_xS$. But by definition $\AA(x)n(x) = 0$, and clearly $(\nabla \pi)(x + tn(x))n(x) = 0$, too. Hence
both sides of \eqref{dpi-1} agree on all of $\R^3$.
\end{proof}

We will frequently extend functions $f : S\to \R$ defined on $S$ only
to functions defined on $S^1$ in the following way:
$$
f(x) = f(\pi(x)) \mbox{ for all }x\in S^1,
$$
with a slight abuse of notation on the left-hand side.
When referring to this extension, we will say that we extend $f$ trivially to $S^1$.
By Lemma \ref{dpi} we have for all $x\in S^1$ the
following formula for the full derivative of $f$ in terms of its tangential derivative:
$$
(\nabla  f)(x) = (\nabla _{tan}f)(\pi(x)) T_S(\pi(x)) (I + t(x) \AA(\pi(x)))^{-1}.
$$
Extending $\nabla_{tan} f$, $T_S$ and $\AA$ trivially from $S$ to $S^1$, as we will do from now on, this formula reads
\begin{equation}
\label{dpi-2}
(\nabla  f)(x) = (\nabla _{tan}f)(x) T_S(x) (I + t(x) \AA(x))^{-1}.
\end{equation}
From now on we tacitly also extend $r$ trivially to $S^1$.

\subsection{Displacements and infinitesimal bendings}

For a given displacement $V : S\to\R^3$ we introduce the quadratic form $(dV)^2$ on $S$ which is defined by
$$
(dV)^2(x)(\tau, \eta) = \dd_{\tau} V(x)\cdot\dd_{\eta} V(x) \mbox{ for all }\tau, \eta\in T_xS.
$$
We also introduce the quadratic form $q_V$ on $S$ which is defined
by its action on tangent vectors $\taub, \etab\in T_{\hx} S$ as follows:
$$
q_V(\hx)(\taub, \etab) = \tfrac{1}{2} \left( \etab\cdot\partial_{\taub} V (\hx)  + \taub\cdot\partial_{\etab} V (\hx) \right)
$$
In the geometry literature, this form is usually denoted by $d\psi\cdot dV$. In local coordinates, it is given by the matrix field
$$
\sym \left( (\nabla\psi^T)\nabla \o{V}\right)
$$
on $\omega$.
Obviously,
\begin{equation}
\label{merd-1}
\sym\left( (\nabla\psi)^T\nabla \o{V}_{tan} \right) = \sym \nabla \bar{V} - \Gamma\cdot \bar{V},
\end{equation}

where $\bar{V}_\alpha=\o{V} \cdot \partial_\alpha \psi$ for $\alpha=1,2$ and
where for brevity we have set $(\Gamma\cdot \bar{V})_{ij} := \sum_{k = 1, 2}\Gamma^k_{ij}\bar{V}_k$.
Here $\Gamma_{ij}^k$ denote the Christoffel symbols of the metric induced by $\psi$. For our purposes it will be enough
to know that $\Gamma\in L^{\infty}(\omega; \R^{2\times 2\times 2})$.
Using \eqref{merd-1} we see that
\begin{equation}
\label{merd-1b}
\sym\left( (\nabla\psi)^T\nabla\o{V} \right) = \sym \nabla \bar{V} - \Gamma\cdot \bar{V} + (\tilde{V}\cdot n)\AA,
\end{equation}
were $\AA$ denotes the pulled back (negated) second fundamental form.
Equivalently, we have the following equality between quadratic forms on $S$:
\begin{equation} \label{eq:identity}
q_V = q_{V_{\ttan}} + (V \cdot \na)\AA.
\end{equation}
It is well-known that the quadratic form $q_V$ typically arises in the context of thin elastic shells,
because it is just the first variation of the metric of $S$ under the displacement $V$.
For example, in \cite{Geymonat-Sanchez-95} it is denoted (in coordinates) by $\gamma_{\alpha\beta}$ and in \cite{Lewicka1} it is denoted
by $\sym\nabla V$.
\\
A displacement $V : S\to\R^3$ is called an {\em infinitesimal bending} of $S$ provided that $q_V = 0$, i.e., that
$$
\sym \left( (\nabla\psi)^T\nabla\o{V} \right)_{ij} = \dd_j\psi\cdot\dd_i\o{V} + \dd_i\psi\cdot\dd_j\o{V} = 0\mbox{ for all }i, j = 1, 2.
$$
Infinitesimal bendings have been studied extensively both in the applied literature (see e.g. \cite{Ciarletshell00}, \cite{Choi-97}, \cite{Geymonat-Sanchez-95}) and in
the geometry literature (see e.g. the references in \cite{Hornung12}).
Recently, they have been found to be relevant as well to fully nonlinear bending theories, cf. \cite{Hornung-CPDE}.
\\
For any displacement $V$,

we define $\mu_V : S\to\R^3$ by setting
\begin{equation}
\label{defmu}
\mu_V = T_S\,\frac{\dd_{\tau_1}V\wedge\tau_2 + \tau_1\wedge\dd_{\tau_2}V}{|\tau_1\wedge\tau_2|}.
\end{equation}
Note that
\begin{equation}
\label{idp-1}
\mu(x)\cdot\tau = - n(x)\cdot\dd_{\tau}V(x) \mbox{ for all }\tau\in T_xS.
\end{equation}
In fact, we compute
\begin{align*}
\mu\cdot\tau &= \frac{1}{|\tau_1\wedge\tau_2|}\left( \dd_{\tau_1}V\cdot\tau_2\wedge\tau - \dd_{\tau_2}V\cdot\tau_1\wedge\tau\right).
\end{align*}
Hence $\mu\cdot\tau_i = - n\cdot\dd_{\tau_i}V$ for $i = 1, 2$. This proves \eqref{idp-1}.
\\
For a given displacement $V : S\to\R^3$ we define $\Omega_V : S\to\R^{3\times 3}$ by
\begin{equation}
\label{defO}
\O_V = \nabla_{tan}V\ T_S + \mu_V\otimes n.
\end{equation}
\begin{lemma}
\label{irre}
If $V\in H^1(S; \R^3)$ then
$
\sym\O_V = q_V\left( T_S, T_S \right)
$
almost everywhere on $S$.
\end{lemma}
\begin{proof}
Clearly $n\cdot \O_V n = n\cdot\mu_V = 0$ and for any tangent vector field $\tau$ along $S$ we have
$$
\tau\cdot \O_V n + n\cdot \O_V \tau = \tau\cdot\mu_V + n\cdot\dd_{\tau}V = 0
$$
by \eqref{idp-1}. For any tangent vector field $\sigma$ we have
$$
\tau\cdot \O_V\sigma + \sigma\cdot \O_V\tau = 2q_V(\sigma, \tau).
$$
\end{proof}

If $V$ is an infinitesimal bending, then
\begin{equation}
\label{form-2}
\O_V^2(T_S, T_S) = -(dV)^2(T_S, T_S),
\end{equation}
that is, $\O_V^2\left( \tau, \sigma \right) = -\dd_{\tau} V\cdot\dd_{\sigma} V$
for all tangent vector fields $\tau$, $\sigma$ along $S$.
\\
In fact, by skew symmetry, $\O_V^2\left( \tau, \sigma \right) = -\O_V\tau\cdot\O_V\sigma$, and
$\dd_{\tau} V = \O_V\tau$.

For any displacement $V$, the {\em linearised Weingarten map} $b_V(x)$ is the linear map on $T_x S$ given by
\begin{equation}
\label{linwein}
b_V = \nabla_{tan}V \AA -\nabla_{tan}\mu_V.
\end{equation}
An infinitesimal bending $V$ determines a linearized second fundamental form $b_V$, which can be regarded as the first order change of the second
fundamental form of the surface $\psi$ under the displacement $V$.
In coordinates, the (negated)
linearized second fundamental form of $V$ is given by
\begin{equation}
\label{defbV}
(b_V)_{ij} = n\cdot (\dd_i\dd_j\o{V} - \Gamma_{ij}^k\dd_k\o{V}),
\end{equation}

 cf. \cite{Hornung12} and
the references therein. The linearized second fundamental form also occurs e.g. in the analysis in \cite{Geymonat-Sanchez-95}.

The following lemma justifies our use of the symbol $b_V$ here.

\begin{lemma}
If $V$ is an infinitesimal bending, then $\dd_i\psi\cdot b_V\dd_j\psi = (b_V)_{ij}$.
\end{lemma}
\begin{proof}
We write $\mu$, $V$, $n$ etc. instead of $\t\mu_V$, $\t V$, $\t n$, and the coordinates of the second fundamental form
are denoted $h_{ij}$, and we use the common convention regarding the raising and lowering of indices.
By definition of the linearised Weingarten map, we have
\begin{align*}
b_V\dd_j\psi &=
- \dd_j\mu + \nabla_{\dd_j n}V
= - \dd_j\mu - h_j^k \dd_k V.
\end{align*}
Hence using $n\cdot\dd_i V + \mu\cdot\d_i\psi = 0$ (which follows from $n\cdot\d_i\psi = 0$), we see
\begin{align*}
\dd_i\psi\cdot b_V\dd_j\psi &=
- \dd_i\psi\cdot \dd_j\mu - h_j^k \dd_i\psi\cdot\dd_k V
\\
&= - \dd_j (\dd_i \psi\cdot \mu) + \dd_j\dd_i\psi\cdot\mu - h_j^k \dd_i\psi\cdot\dd_k V
\\
&= \dd_j (\dd_i V\cdot n) + \Gamma_{ij}^k\d_k\psi\cdot\mu - h_j^k \dd_i\psi\cdot\dd_k V
\\
&= \dd_j\dd_i V\cdot n + \dd_i V\cdot \dd_j n - \Gamma_{ij}^k\d_k V\cdot n - h_j^k \dd_i\psi\cdot\dd_k V
\\
&= n\cdot\left( \dd_j\dd_i V - \Gamma_{ij}^k\d_k V  \right) - h_j^k \dd_i V\cdot \dd_k \psi - h_j^k \dd_i\psi\cdot\dd_k V
\end{align*}
This indeed agrees with $(b_V)_{ij}$ as defined in \eqref{defbV}, because the last terms cancel by the definition of infinitesimal bendings.
\end{proof}

We will frequently need the following diffeomorphism $\Phi^h : S^h \to S^1$:
$$
\Phi^h(x) = \pi(x) + \frac{t(x)}{h}n(x).
$$

The following lemma summarizes a computation that will later be used for the generic type of ansatz functions.

\begin{lemma}
\label{ansatz}
Let $h\in (0, 1/2)$, let $V\in H^2(S; \R^3)$, and for $x\in S^h$ define
$$
\rho(x) = V(x) + t(x)\mu_V(x).
$$
Then the following equality holds on $S^h$:
\begin{equation}
\label{Drho}
\begin{split}
\nabla\rho = \O_V &- t b_V(T_S, T_S) - t^2\nabla_{tan}\mu_V \AA
\\
& + \left( \nabla_{tan}V + t\nabla_{tan}\mu_V \right)T_S
\left( \left(\I + t\AA\right)^{-1} - \left( \I-t\AA \right) \right),
\end{split}
\end{equation}
where we extend $V$, $\mu_V$, $\O_V$, $b_V$, $\nabla_{tan} V$ etc. trivially from $S$ to $S^h$.
\end{lemma}
\begin{proof}
For all $x\in S^h$ define
\begin{equation}
\label{defQ}
Q(x) = (\I+t(x)\AA(x))^{-1} - \left( \I-t(x)\AA(x) \right).
\end{equation}
Since clearly $\nabla t = n$, formula \eqref{dpi-2} shows that on $S^h$:
\begin{align*}
\nabla\rho &= \left( \nabla_{tan}V + t\nabla_{tan}\mu_V \right) T_S \left(I + t\AA\right)^{-1}
+ \mu_V\otimes n
\\
&= \left( \nabla_{tan}V + t\nabla_{tan}\mu_V \right) T_S \left(I - t\AA\right)
+ \mu_V\otimes n + \left( \nabla_{tan}V + t\nabla_{tan}\mu_V \right) T_S Q
\\
&= \nabla_{tan}V + \mu_V\otimes n - t\nabla_{tan}V\AA + t\nabla_{tan}\mu_V T_S
- t^2\nabla_{tan}\mu_V\AA + \left( \nabla_{tan}V + t\nabla_{tan}\mu_V \right) T_S Q.
\end{align*}
By the definition of $\Omega_V$ and $b_V$ this is the claim.
\end{proof}

\section{Elasticity framework and main result} \label{treci}

Throughout this paper we assume that the limit
$$
\gamma := \lim_{h \to 0} \tfrac{h}{\eh}
$$
exists in $[0,\infty]$. We will frequently write $\e$ instead of $\eh$, but always with the understanding that $\e$ depends on $h$ via $\gamma$.

\begin{definition}[nonlinear material law]\label{def:materialclass}
  Let $0<\alpha\leq\beta$ and $\rho>0$. The class
  $\calW(\alpha,\beta,\rho)$ consists of all measurable functions
  $W\,:\,\R^{3 \times 3}\to[0,+\infty]$ that satisfy the following properties:
  \begin{align}
    \tag{W1}\label{ass:frame-indifference}
    &W\text{ is frame indifferent, i.e.}\\
    &\notag\qquad W(\Ra\boldsymbol F)=W(\boldsymbol F)\quad\text{ for
      all $\boldsymbol F\in\R^{3 \times 3}$, $\boldsymbol R\in\SO
      3$;}\\
    \tag{W2}\label{ass:non-degenerate}
    &W\text{ is non degenerate, i.e.}\\
    &\notag\qquad W(\boldsymbol F)\geq \alpha\dist^2(\boldsymbol F,\SO 3)\quad\text{ for all
      $\boldsymbol F\in\R^{3 \times 3}$;}\\
    &\notag\qquad W(\boldsymbol F)\leq \beta\dist^2(\boldsymbol F,\SO 3)\quad\text{ for all
      $\boldsymbol F\in\R^{3 \times 3}$ with $\dist^2(\boldsymbol F,\SO 3)\leq\rho$;}\\
    \tag{W3}\label{ass:stressfree}
    &W\text{ is minimal at $\id$, i.e.}\\
    &\notag\qquad W(\id)=0;\\
    \tag{W4}\label{ass:expansion}
    &W\text{ admits a quadratic expansion at $\id$, i.e.}\\
    &\notag\qquad W(\id+\boldsymbol G)=\mathcal{Q}(\boldsymbol G)+o(|\boldsymbol G|^2)\qquad\text{for all }\boldsymbol G\in\R^{3 \times 3}\\
    &\notag\text{where $\mathcal{Q} \,:\,\R^{3 \times 3}\to\R$ is a quadratic form.}
  \end{align}
\end{definition}

\begin{definition}[admissible composite material]\label{def:composite}
  Let $0<\alpha\leq\beta$ and $\rho>0$. We say
  \begin{equation*}
    W:S^{1}\times\R^2\times \R^{3 \times 3}\to \R^+\cup\{+\infty\}
  \end{equation*}
  describes an admissible composite material of class $\mathcal W(\alpha,\beta,\rho)$ if
  \begin{enumerate}[(i)]
  \item $W$ is almost everywhere equal to a Borel function on $S^{1} \times\R^2\times\R^{3 \times 3}$,
  \item $W(\cdot,y,\boldsymbol F)$ is continuous for almost every $y\in\R^2$ and $\boldsymbol F\in\R^{3 \times 3}$,
  \item $W(x,\cdot,\boldsymbol F)$ is $Y$-periodic for all $x\in\Omega$ and almost every $\boldsymbol F\in\R^{3 \times 3}$,
  \item $W(x,y,\cdot)\in\mathcal W(\alpha,\beta,\rho)$ for all $x\in S^{1}$ and almost every $y\in\R^2$.

  \end{enumerate}
\end{definition}
\begin{assumption}
  \label{ass:main}
  We assume that
  \begin{itemize}
  \item $W$ describes an admissible composite material of class $\mathcal
    W(\alpha,\beta,\rho)$ in the sense of Definition~\ref{def:composite}.
  \item $\mathcal{Q}$ is the quadratic energy density associated to $W$
    through expansion \eqref{ass:expansion} in Definition~\ref{def:materialclass}.
  \item The following uniformity is valid
  $$ \lim_{G \to 0} \esssup_{(x,y) \in S^1 \times \calY} \frac{|W(x,y,I+G)-\mathcal{Q}(x,y,G)|}{|G|^2}=0. $$
  \end{itemize}
\end{assumption}

We collect some basic properties of admissible $W$ and the associated quadratic forms $\mathcal{Q}$;
a proof can be found in \cite[Lemma~2.7]{Neukamm-11}.

\begin{lemma}
  \label{lem:1}
  Let $W$ and $\mathcal{Q}$ satisfy the assumption (\ref{ass:main}). Then
  \begin{enumerate}[(i)]
  \item[(Q1)] $\mathcal{Q}(\cdot,y,\cdot)$ is continuous for almost every $y\in\R^2$,
  \item[(Q2)] $\mathcal{Q}(x,\cdot,\boldsymbol G)$ is $Y$-periodic
      and measurable for all $x\in S^{1}$ and all
      $\boldsymbol F\in\R^{3 \times 3}$,
  \item[(Q3)] for all $x\in S^{1}$ and almost every
      $y\in\R^2$ the map $\mathcal{Q}(x,y,\cdot)$ is quadratic and
      satisfies
    \begin{equation*}
      \alpha|\sym \boldsymbol G|^2\leq \mathcal{Q}(x,y,\boldsymbol G)=\mathcal{Q}(x,y,\sym \boldsymbol G)\leq \beta|\sym \boldsymbol G|^2\qquad\text{for all $ \boldsymbol G\in\R^{3 \times 3}$.}
    \end{equation*}
  \end{enumerate}
  Furthermore,  there exists a monotone function $m:\R^+\to\R^+\cup\{+\infty\}$ such that $m(\delta)\to 0$ as
  $\delta\to 0$ and
  \begin{equation}\label{eq:94}
    \forall \boldsymbol G\in\R^{3 \times 3}\,:\,|W(x,y,\id+\boldsymbol G)-\mathcal{Q}(x,y,\boldsymbol G)|\leq|\boldsymbol G|^2m(|\boldsymbol G|)
  \end{equation}
  for all $x\in S^{1}$ and almost every $y\in\R^2$.
\end{lemma}

Let $W$ be an energy density satisfying Assumption \ref{ass:main}.
The elastic energy per unit thickness
of a deformation $u^h \in H^1(S^h;\R^3)$ of the shell $S^h$ is given by
$$
E^h(u^h)=\frac{1}{h}\int_{S^h} W\left( \Phi^h(x), r(x)/\e, \nabla u^h(x) \right)\ dx.
$$

We denote by $\mathcal{B}$ the $L^2$-closure of the set
$$
\{ q_w : w \in H^1 (S;\R^3) \}.
$$
As this is a linear space, its strong and its weak $L^2$-closure coincide.
The set
$\mathcal{B}$ is a closed linear subspace of $L^2(S; \QQ)$.
The space $\mathcal{B}$ is also encountered in the context of shell models derived from linearized elasticity;
see \cite{Sa89,Sanchez-89,Geymonat-Sanchez-95} for details.

Before we give the main statement we have to define the limit functionals. To do that we need the definition of the relaxation fields and the cell formulas.

\begin{definition}\label{def:relfield}
  We define the following operators:
  \begin{eqnarray*}
 && \mathcal{U}_0: \dot{H}^1(\mathcal Y;\R^2)  \times  L^2(I\times\mathcal Y;\R^3) \to L^2 (I \times\mathcal Y;\R^{3 \times 3}_{\sym}), \\
  && \mathcal{U}_0 (\zetaa,\bs g) = \left(\begin{array}{cc}
        \sym \nabla_y\zetaa
        & \begin{array}{cc} \bs g_1\\ \bs g_2\end{array}\\
        (\bs g_1,\;\bs g_2)&\bs g_3\\
      \end{array}\right)_{ij}\tau^i\otimes\tau^j, \\ && \\
  && \mathcal{U}_0^0: \dot{H}^1(\mathcal Y;\R^2) \times
    \dot{H}^2(\mathcal
    Y) \times  L^2(I\times\mathcal Y;\R^3) \to L^2 (I \times\mathcal Y;\R^{3\times 3}_{\sym}), \\
  && \mathcal{U}_0^0 (\zetaa,\varphi,\bs g) =
  \left(\begin{array}{cc}
        \sym\nabla_y\zetaa-t \nabla^2_y \varphi
        & \begin{array}{cc} \bs g_1\\ \bs g_2\end{array}\\
        (\bs g_1,\;\bs g_2)&\bs g_3\\
      \end{array}\right)_{ij}\ \tau^i\otimes\tau^j, \\ && \\
   && \mathcal{U}_{0,\gamma_1}^{1}: \dot{H}^1(\mathcal Y;\R^2) \times
    \dot{H}^2(\mathcal
    Y) \times  L^2(I\times\mathcal Y;\R^3) \to L^2 (I \times\mathcal Y;\R^{3 \times 3}_{\sym}), \\
  && \mathcal{U}_{0,\gamma_1}^{1} (\zetaa,\varphi,\bs g) =
 \left(\begin{array}{cc}
        \sym\nabla_y\zetaa+\tfrac{1}{\gamma_1}\varphi\AA(\hx)-t\hat \nabla^2_y \varphi
        & \begin{array}{cc} \bs g_1\\ \bs g_2\end{array}\\
        (\bs g_1,\;\bs g_2)&\bs g_3\\
      \end{array}\right)_{ij}\ \tau^i\otimes\tau^j, \\
&& \mathcal{U}_\infty:L^2(I;\dot{H}^1(\mathcal Y;\R^2))\times L^2(I; \dot{H}^1(\mathcal Y)) \times L^2(I;\R^3) \to L^2 (I \times\mathcal Y;\R^{3 \times 3}_{\sym}), \\
  && \mathcal{U}_\infty(\zetaa,\psi,\bs c)=
    \left(\begin{array}{cc}
        \sym\nabla_y\zetaa
        & \begin{array}{cc} \partial_{y_1}\psi+\bs
          c_1\\ \partial_{y_2}\psi+\bs c_2\end{array}\\
        \nabla_y\psi +(\bs c_1,\bs c_2) & \bs c_3
      \end{array}\right)_{ij}\tau^i\otimes\tau^j,\\ && \\
       && \mathcal{U}_{\gamma} : \dot{H}^1(I\times\mathcal Y;\R^3) \to  L^2 (I \times\mathcal Y;\R^{3 \times 3}_{\sym}); \text{ for } \gamma \in (0,\infty); \\
       && \mathcal{U}_\gamma (\phia) =
\sym(\nabla_y\phia,\tfrac{1}{\gamma}\partial_3\phia)_{ij}\ \tau^i\otimes\tau^j.
  \end{eqnarray*}

For $\gamma\in (0, \infty)$ we introduce the function spaces of \textit{relaxation fields}
$$
L_{\gamma}(I\times\mathcal Y) =
\left\{ \mathcal{U}_\gamma(\phi):
\phi\in\dot H^1(I\times\mathcal Y; \R^3) \right\}
$$
For $\gamma=\infty$ and $\gamma=0$ we define
\begin{align*}
L_\infty(I \times \calY) =&\left\{ {U}_\infty(\zetaa,\psi,\bs c):\, (\zeta,\psi,c) \in L^2(I;\dot{H}^1(\mathcal Y;\R^2))\times L^2(I; \dot{H}^1(\mathcal Y)) \times L^2(I;\R^3) \right\}, \\
L_0 (I\times\mathcal Y)= & \left\{ {U}_0(\zetaa,g):\, (\zeta,g) \in \dot{H}^1(\mathcal Y;\R^2)  \times  L^2(I\times\mathcal Y;\R^3) \right\}.
\end{align*}
We also introduce
\begin{align*}
L_0^0(I\times\mathcal Y) = &
\Big\{ \left(\begin{array}{cc}
        \sym\nabla_y\zetaa-t \nabla^2_y \varphi
        & \begin{array}{cc} \bs g_1\\ \bs g_2\end{array}\\
        (\bs g_1,\;\bs g_2)&\bs g_3\\
      \end{array}\right)_{ij}\ \tau^i\otimes\tau^j :
 \zeta\in \dot{H}^1(\mathcal Y;\R^2),\\  & \hspace{+10ex}  \varphi\in\dot{H}^2(\mathcal Y),\ g\in L^2(I\times\mathcal Y;\R^3)
\Big\}
\end{align*}
and for $x \in S$, $\gamma_1 \in (0,\infty)$ we define
\begin{align*}
L_{0,\gamma_1}^{1} (I \times \calY)= &
\Big\{
\left(\begin{array}{cc}
        \sym\nabla_y\zetaa+\tfrac{1}{\gamma_1}\varphi\AA(\hx)-t \nabla^2_y \varphi
        & \begin{array}{cc} \bs g_1\\ \bs g_2\end{array}\\
        (\bs g_1,\;\bs g_2)&\bs g_3\\
      \end{array}\right)_{ij}\ \tau^i\otimes\tau^j :
 \zeta\in \dot{H}^1(\mathcal Y;\R^2),\\ & \varphi\in\dot{H}^2(\mathcal Y),\ g\in L^2(I\times\mathcal Y;\R^3)
\Big\}.
\end{align*}
\end{definition}
\begin{remark}
Notice that all the operators $\mathcal{U}$ and the appropriate spaces also depend on $x \in S$. For simplicity of writing we do not write $x$ in the notation.
\end{remark}
For $\gamma \in (0,\infty]$ and $x \in S$ we define the functions $\mathcal{Q}_{\gamma} (x) : \QQ(x)_{\sym} \times \QQ(x)_{\sym} \to \R$ as follows:
\begin{equation} \label{defQg}
 \mathcal{Q}_{\gamma}(x, q^1, q^2) = \inf_{\bs U \in L_{\gamma} (I \times \calY ) }
\int_I\int_{\calY} \mathcal{Q}\Big(x + tn(x), y, q^1 + tq^2 + U \Big) \ud y \ud t.
\end{equation}
For  $\gamma_1 \in (0,\infty)$ we define $\mathcal{Q}_0^0(x), Q_{0,\gamma_1}^{1}(x): \QQ(x)_{\sym} \times \QQ(x)_{\sym} \to \R$ as follows:
\begin{align} \label{defQgi}
\mathcal{Q}_0^0(x, q^1, q^2) &= \inf_{\bs U \in L_0^0 (I \times \calY ) }
\int_I\int_{\calY} \mathcal{Q}\Big(x + tn(x), y, q^1 + tq^2 + U \Big) \ud y \ud t
\\
\label{defQg01}
\mathcal{Q}_{0,\gamma_1}^{1} (x,q^1, q^2) &=
\inf_{\bs U \in L_{0,\gamma_1}^{1} (I \times \calY ) }
\int_I\int_{\calY} \mathcal{Q}\Big(x + tn(x), y, q^1 + tq^2 + U \Big) \ud y \ud t
\end{align}
\begin{remark}
We discuss the cell formula in the limiting cases $\gamma=0$ and $\gamma=\infty$.
\begin{enumerate}[(i)]
\item In the case $\gamma=\infty$ define
 \begin{eqnarray*}
 \widetilde{L}_\infty(I\times\mathcal Y)&:=&\Bigg\{\,  \sum_{i,j=1}^3 \left(\begin{array}{cc}
        \sym\nabla_y\zetaa
        & \begin{array}{cc} \partial_{y_1}\psi+\bs
          c_1\\ \partial_{y_2}\psi+\bs c_2\end{array}\\
        \nabla_y\psi +(\bs c_1,\bs c_2) & \bs c_3
      \end{array}\right)_{ij} \taub^i \otimes \taub^j\,:
     \\
    &&\qquad \zetaa \in \dot{H}^1(\mathcal Y,\R^2),\psi \in  \dot{H}^1(\mathcal Y),\,\bs c\in \R^3  \Bigg\}.
\end{eqnarray*}
Also define for $(x,t) \in S\times I$
\begin{equation}
\widetilde{\mathcal{Q}}_\infty(x,t,q^1,q^2)=  \inf_{\bs U \in \widetilde{L}_\infty (I \times \calY; \R^{3 \times 3}_{\sym} ) }
\int_{\calY}   \mathcal{Q} \Big(x+tn(x),y, q^1+tq^2+\bs U \Big) \ud y.
\end{equation}
It is easy to see that $\widetilde{\mathcal{Q}}_\infty$ is, for a fixed $x\in S$, $t \in I$, a quadratic in $q^1,q^2$. We have
\begin{equation} \label{eq:iggg1112}
\mathcal{Q}_\infty(x,q^1,q^2)=\int_I \widetilde{\mathcal{Q}}_\infty(x+tn(x),q^1,q^2) \ud t
\end{equation}

\item Define as in \cite{Lewicka1}
\begin{equation}
\mathcal{Q}_2(x,t,q^1,q^2)=\min_{\Ma \in \R^{3 \times 3}_{\sym}} \{ \mathcal{Q}(x+tn(x),\Ma): q^1 + tq^2-\sum_{i,j=1,2} (\Ma\taub_j \cdot \taub_i)\taub^i \otimes \taub^j=0 \}.
\end{equation}
Also define
 \begin{eqnarray*}
    \widetilde{L}^0_0(I\times\mathcal Y) &=& \Bigg\{\, \sum_{i,j=1}^2
\left(\begin{array}{c}
        \sym\nabla_y\zetaa-t \nabla_y^2\varphi
      \end{array}\right)_{ij} \taub^i \otimes \taub^j\,: \zetaa\in \dot{H}^1(\mathcal Y,\R^2),
   \varphi\in \dot{H}^2(\mathcal
    Y) \Bigg\}, \\
    \widetilde{L}_{0,\gamma_1}^{1}(I\times\mathcal Y)&=&\Bigg\{\, \sum_{i,j=1}^2\left(\begin{array}{c}
        \sym\nabla_y\zetaa+\tfrac{1}{\gamma_1} \varphi\AA_{ij}(\hx)-t\nabla_y^2\varphi
      \end{array}\right)_{ij} \taub^i \otimes \taub^j\,:
     \\  &&\qquad \zetaa\in \dot{H}^1(\mathcal Y,\R^2),
   \varphi\in \dot{H}^2(\mathcal
    Y) \Bigg\}
\end{eqnarray*}
It can be easily seen that we have for the cell formula
\begin{eqnarray} \label{eq:iggg1111}
& & \mathcal{Q}_0^0(x,q^1,q^2)= \inf_{\bs U \in \widetilde{L}_0^0 (I \times \calY) }
\iint_{I \times \mathcal{Y}} \mathcal{Q}_2\Big(x+tn(x),y, q^1 + tq^2+\bs U \Big) \ud t \ud y,
\end{eqnarray}
for i=0,2 i.e.
\begin{eqnarray} \label{eq:iggg1111111}
 \mathcal{Q}_{0,\gamma_1}^{1} (x,q^1,q^2)&=& \\ \nonumber && \hspace{-10ex} \inf_{\bs U \in \widetilde{L}_{0,\gamma_1}^{1} (I \times \calY ) }
\iint_{I \times \mathcal{Y}} \mathcal{Q}_2\Big(x+tn(x),y, q^1+tq^2+\bs U \Big) \ud t \ud y.
\end{eqnarray}
In the case when $\mathcal{Q}$ does not depend on $t$ we have that
\begin{eqnarray*}
\mathcal{Q}_0^0(x,q^1,q^2) &=& \inf_{\zetaa \in \dot{H}^1(\mathcal Y,\R^2)  }\int_{\mathcal{Y}} \mathcal{Q}_2(x,y,q^1+\sum_{i,j=1}^2(\sym\nabla_y\zetaa)_{ij}\taub^i \otimes \taub^j ) \ud y\\ \nonumber & & + \frac{1}{12} \inf_{\varphi \in \dot{H}^2 (\mathcal Y)}
\int_{\mathcal{Y}} \mathcal{Q}_2 (x,y,q^2+\sum_{i,j=1}^2(\nabla^2_y\varphi)_{ij}\taub_i \otimes \taub_j ) \ud y, \\
\mathcal{Q}_0^{1,\gamma_1} (x,q^1,q^2) &=&\\ \nonumber
 \inf_{\zetaa \in \dot{H}^1(\mathcal Y,\R^2),\varphi \in \dot{H}^2 (\mathcal Y)  }&&\bigg(\int_{\mathcal{Y}} \mathcal{Q}_2\Big(x,y,q^1+\sum_{i,j=1}^2\big((\sym\nabla_y\zetaa)_{ij}+\tfrac{1}{\gamma_1}
\varphi\AA_{ij} \big)\taub^i \otimes \taub^j\Big) \ud y\\ \nonumber & & + \frac{1}{12}
\int_{\mathcal{Y}} \mathcal{Q}_2 (x,y,q^2+\sum_{i,j=1}^2(\nabla^2_y\varphi)_{ij}\taub_i \otimes \taub_j ) \ud y\bigg).
\end{eqnarray*}

\end{enumerate}

\end{remark}

\begin{remark}
In the same way as in \cite{NeuVel-12} we can prove the following: For every $q^1, q^2 \in \mathbb{S}(x)_{\sym}$ and $\hx \in S$ we have that
\begin{eqnarray*} \lim_{\gamma \to \infty} \mathcal{Q}_{\gamma}(\hx,q^1, q^2)&=& \mathcal{Q}_{\infty}(\hx,q^1, q^2)\\  \lim_{\gamma \to 0} \mathcal{Q}_{\gamma}(\hx,q^1, q^2) &=& \mathcal{Q}_0^0 (\hx,q^1, q^2)
\end{eqnarray*}
\end{remark}

\begin{remark}
Notice that when $\AA=0$ then
all spaces $L_0^0$ and $L_{0,\gamma_1}^{1}$ coincide for $\gamma_1 \in (0,\infty)$.
This corresponds to the observation in the von K\'arm\'an plate theory that
for $\gamma=0$ one obtains only one relaxation space, cf. \cite{NeuVel-12} for details.
\end{remark}

For $\gamma \in (0,\infty]$ define the functionals $I_\gamma : H^2(S;\R^3) \times L^2(S;\QQ)\to \R$ by setting
\begin{equation} \label{eq:iggg22211}
I_\gamma (\V, \B_w)=\int_S \mathcal{Q}_\gamma(\cdot,\B_w + \tfrac{1}{2} (dV)^2,\, - \bbv \big) d\HH^2,
\end{equation}
and define the functionals $I_0^0:H^2(S;\R^3) \times L^2(S;\QQ)\to \R$ by
\begin{equation} \label{eq:iggg222211}
I_0^0 (\V,\B_w)=\int_S \mathcal{Q}_0^0(\cdot,\B_w + \tfrac{1}{2} (dV)^2,\, - \bbv \big) d\HH^2,
\end{equation}
as well as, for $\gamma_1 \in (0,\infty)$, define the functionals $I_{0,,\gamma_1}^{1}:H^2(S;\R^3) \times L^2(S;\QQ)\to \R$ by
\begin{equation} \label{eq:iggg2222211}
I_{0,\gamma_1}^{1}  (\V,\B_w)=\int_S \mathcal{Q}_{0,\gamma_1}^{1} (\cdot,\B_w + \tfrac{1}{2} (dV)^2,\, - \bbv \big) d\HH^2.
\end{equation}

This is our main result:

\begin{theorem} \label{tm:glavni}
Let $W$ satisfy Assumption \ref{ass:main} and assume that
$\ua^h \in H^1(S^h;\R^3)$ satisfy
\begin{equation} \label{uuuvjet}
\limsup_{h \to 0} h^{-4}E^h(\ua^h) < \infty.
\end{equation}

Then the following are true:
\begin{enumerate}[(i)]
\item \textit{(compactness)}.
There exists a subsequence, still denoted by $(\bar{\ya}^h)$,
and there exist $Q^h\in \SO 3$ and $c^h\in\R^3$ such that
the sequences $y^h$ and $V^h$ defined by
$$
\ya^h = (\bs Q^h)^T\bar{\ya}^h-\bs c^h
$$
and
$$
\V^h(x) = \frac{1}{h} \left(\int_I \ya^h(\hx+t\na(\hx)) dt\ - x\right) \mbox{ for all }x\in S
$$
satisfy the following:
\begin{enumerate}[(a)]
    \item We have
$$
\ya^h \to \pi\mbox{ strongly in }H^1(S^{1};\R^3).
$$
    \item There exists an infinitesimal bending $\V \in H^2(S;\R^{3})$ of $S$ such that
$$
V^h \to \V\mbox{ strongly in }H^1(S;\R^3).
$$
\item There exists $B_w\in L^2(S; \QQ)$ such that
$$
\frac{1}{h} q_{\V^h} \rightharpoonup B_w\mbox{ weakly in }L^2(S;\QQ).
$$
\end{enumerate}
\item
\textit{(lower bound)}. Defining
$I_\gamma$ by (\ref{eq:iggg22211}) and $I_0^0$ by (\ref{eq:iggg222211}) and $I_0^{1,\gamma_1}$ by (\ref{eq:iggg2222211}), we have

$$
\liminf_{h \to 0} h^{-4} E^h (\ua^h) \geq
\begin{cases}
I_\gamma (\V,\B_w) &\mbox{ if }h/\e\to \gamma \in (0,\infty]
\\
I_0^0 (\V,\B_w) &\mbox{ if }\e\gg h\gg \e^2
\\
I_{0,\gamma_1}^{1} (\V,\B_w) &\mbox{ if }\e^2/h \to \tfrac{1}{\gamma_1} \in (0,\infty)
\end{cases}
$$
\item \textit{(recovery sequence)} For any infinitesimal bending $\V \in H^2(S,\R^3)$ of $S$
and any $\B_w \in \mathcal{B}$, there exist $\ua^h\in H^1(S^h;\R^3)$
satisfying (\ref{uuuvjet}), and such that the conclusions of part (i) are true
with $\bs Q^h=\I$ and $\bs c^h=0$, and
$$
\lim_{h \to 0} h^{-4} E^h (\ua^h) =
\begin{cases}
I_\gamma (\V,\B_w) &\mbox{ if }h/\e\to \gamma \in (0,\infty]
\\
I_0^0 (\V,\B_w) &\mbox{ if }\e\gg h\gg \e^2
\\
I_{0,\gamma_1}^{1} (\V,\B_w) &\mbox{ if }\e^2/h \to \tfrac{1}{\gamma_1} \in (0,\infty).
\end{cases}
$$
\end{enumerate}
\end{theorem}

From now on $u^h\in H^1 (S^h;\R^3)$ will always denote a sequence satisfying \eqref{uuuvjet}.

\subsection{Unit thickness rescaling}

Recall that $\Phi^h : S^h \to S^1$ is given by
$$
\Phi^h(x) = \pi(x) + \frac{t(x)}{h}n(x).
$$
Since $\nabla t = n$, Lemma \ref{dpi} and formula \eqref{dpi-2} show (recall that $n$ is extended trivially to $S^1$):
\begin{align*}
\nabla\Phi^h &= \nabla\pi + \frac{t}{h}\nabla n + \frac{1}{h}n\otimes n
\\
&= T_S(I + t\AA)^{-1} + \frac{t}{h}\AA T_S(I + t\AA)^{-1} + \frac{1}{h}n\otimes n
\end{align*}
Since $T_S$ clearly commutes with $\AA$, we see that $T_S$ commutes with $(I + t\AA)^{-1}$ as well. Hence
\begin{equation}
\label{dphi}
\nabla\Phi^h = (I_h + \frac{t}{h}\AA)(I + t\AA)^{-1}\mbox{ on }S^h,
\end{equation}
where $I_h=T_S+\tfrac{1}{h} n \otimes n$.
Following \cite{FJM-02}, for given $u : S^h\to\R^3$ we define its rescaled version $y : S^1\to\R^3$ by
$$ 
\ya(\Phi^h) = \ua \mbox{ on }S^h.
$$ 
We define the rescaled gradient of $y$ by the condition
\begin{equation}
\label{defDh}
\nabla_h y( \Phi^h ) = \nabla u \mbox{ on }S^h.
\end{equation}
To compute $\nabla_h$ more explicitly, insert the definition of $y$ into \eqref{defDh} and use \eqref{dphi} to find
\begin{equation}
\label{fladh}
\nabla_h y =  \nabla y\ (I_h + t\AA)(I + ht\AA)^{-1} \mbox{ on }S^1,
\end{equation}

In order to express the elastic energy in terms of the new variables, we associate with $y : S^1\to\R^3$
the energy
\begin{eqnarray*}
I^h(y) &=& \int_{S^1} W\left( x, r(x)/\e, \nabla_h y(x) \right)\det\left(I+t(x) \AA(x) \right)^{-1}\ dx \\
&=& \int_S \int_I W\left( x+t n(x), r(x)/\e, \nabla_h y(x) \right)\ dt\ d\mathcal{H}^2\\.
\end{eqnarray*}
By a change of variables we have
$$
E^h(u^h) = \frac{1}{h} \int_{S^1} W\left( \cdot, r/\e, \nabla_h y^h \right)\ \left|\det\nabla (\Phi^h)^{-1}\right|.
$$
Using \eqref{dphi} it is easy to see  that
$$
E^h(u^h) = I^h(y^h) + o(h^4)\mbox{ as }h\to 0.
$$

\subsection{FJM-compactness}

The following lemma proves the first part of Theorem \ref{tm:glavni}. It
is a direct consequences of \cite[Theorem 3.1]{FJM-02} and of arguments in
\cite{FJM-06}. We refer to \cite{Lewicka1} for the extension to the present setting.

\begin{lemma}\label{lm:100}
There exist a constant $C>0$, independent of $h$, and a sequence of
matrix fields $(\Ra^h) \subset H^1(S; SO(3))$ (extended trivially to $S^h$)
and there exists a sequence of matrices $(\Q^h) \subset \SO 3$ such that:
\begin{enumerate}[(i)]
\item $\limsup_{h \to 0} h^{-5/2}\| \nabla \ua^h - \Ra^h\|_{L^2(S^h)} <\infty$
\item $\limsup_{h\to 0} h^{-1}\|\nabla \Ra^h\|_{L^2(S)}<\infty$
\item $\limsup_{h \to 0} h^{-1}\|(\Q^h)^T \Ra^h-\I\|_{L^p(S)} < \infty$, for all
    $p \in [1,\infty)$.
\item $(\Q^h)^T \Ra^h \to \I$ strongly in $H^1$.
\end{enumerate}
Moreover, there exists a matrix field $\A\in H^1(S, \so 3)$ taking values in the space of skew symmetric matrices,
such that (after passing to subsequences)
\begin{enumerate}[(i)]
\setcounter{enumi}{4}
\item $ \tfrac{1}{h} \big( (\Q^h)^T \Ra^h-\I
    \big) \rightharpoonup \A,$ weakly in $H^1(S;\R^{3 \times 3})$.
\item $\tfrac{1}{h^2} \sym\big( (\Q^h)^T
    \Ra^h-\I \big) \to \tfrac{1}{2} \A^2,\ \textrm{ strongly in
    } L^p(S;\R^{3 \times 3})$, \ for all
    $p \in [1,\infty)$.
\end{enumerate}
Moreover, the following are true:
\begin{enumerate}[(i)]
\item $\limsup_{h \to 0} \tfrac{1}{h^2} \| \nabla_h \bar{\ya}^h-\Ra^h \|_{L^2(S^1)}
    < \infty$.
\item $ \tfrac{1}{h} \big( (\Q^h)^T \nabla_h
    \ya^h-\I \big) \rightharpoonup \A$,\ \textrm{ weakly in } $H^1$ up
    to a subsequence.
\end{enumerate}
Define $\ya^h \in H^1(S^{1};\R^3)$ by
$$
\ya^h=(\Q^h)^T \bar{\ya}^h-\ca^h,
$$
where
$$
\ca^h = \fint_S \int_I \big((\Q^h)^T \bar{\ya}^h(x + tn(x)) - x\big) dt\ d\HH^2(x).
$$
Introduce the (average) midplane displacements $V^h : S\to\R^3$ by setting
\begin{equation}\label{eq:201}
  \V^h(x) := \frac{1}{h} \left(\int_I \big( \ya^h(x + t\na(x)) \big) dt - x\right)\mbox{ for all }x\in S.
\end{equation}
Then $\fint_S \V^h=0$ and (after passing to a subsequence)
\begin{enumerate}[(i)]
\setcounter{enumi}{2}
\item $ \ya^h \to \pi$, strongly in
    $H^1(S^{1};\R^3)$.
\item There exists an infinitesimal bending $V\in H^2(S; \R^3)$ of $S$ with $\O_V = A$ and such that
$
\V^h \to \V
$
strongly in $H^1(S;\R^3)$.
\item $\tfrac{1}{h} q_{V^h}$ is bounded in $L^2(S;\R^{3 \times 3})$.
\end{enumerate}
\end{lemma}

In what follows we replace the sequence $\Ra^h$ by $(\Q^h)^T \Ra^h$ and the sequence $\ya^h$ by $\bar{\ya}^h$, so
we assume without loss of generality that $Q^h = Id$.
\\
Expressed in the unrescaled variables, we have
$$
V^h(x) = \frac{1}{h^2} \left(\int_{I^h} u^h(x + tn(x))\, dt\, - x\right),
$$
i.e. $x + hV^h(x) = \fint_{I^h} u^h(x + tn(x))\ dt$.

Next we modify the displacement fields $\V^h$ into more regular fields $\V^h_s$ enjoying a similar compactness.

\begin{lemma} \label{kor:1}
There exist $\V_s^h \in H^2(S; \R^3)$ with $\fint_S \V_s^h = 0$ satisfying
\begin{equation}
\label{kor-b}
\limsup_{h \to 0} h^{-1}\| \V_s^h-\V^h\|_{H^1(S)} <\infty
\end{equation}
and
\begin{equation}
\label{kor-a}
\left\|\left( \nabla_{tan} V^h_s - \frac{R^h - I}{h} \right)T_S\right\|_{L^2(S)} \leq
\left\|\left( \nabla_{tan} V^h - \frac{R^h - I}{h} \right)T_S\right\|_{L^2(S)}.
\end{equation}
Moreover, $(V^h_s)$ is uniformly bounded in $H^2(S)$ and
\begin{equation}
\label{kor-c}
V_s^h \weak V\mbox{ weakly in }H^2(S;\R^3).
\end{equation}
\end{lemma}
\begin{proof}
We follow \cite[Proposition~3.1]{NeuVel-12}.
For $i=1,2,3$ denote by $p_i$ the $i$-th row of the matrix $\frac{\Ra^h(\psi) - \I}{h} \nabla\psi$.
We define $\tilde \V^h_s \in H^2(\tilde \omega;\R^3)$ such that $(\tilde \V^h_s)_i$ is a minimiser of the functional
\begin{equation*}
 v\mapsto\int_{\tilde\omega}|\nabla v - p_i|^2\, dx
\end{equation*}
among all $v\in H^1(\tilde \omega)$ satisfying $\int_\omega v=0$, and we define $V^h_s$ via $V^h_s(\psi) = \tilde V^h_s$.
The bound \eqref{kor-a} follows from the minimality of $V^h_s$.
Combining the tangential components of \eqref{tay-1} and \eqref{tay-2} below, we obtain
$$
\|\nabla_{tan} V^h_s - \nabla_{tan} V^h\|_{L^2(S)} \leq Ch.
$$
Hence \eqref{kor-b} follows from Poincar\'e's inequality on $S$.
\\
Since $\partial\tilde \omega$ is $C^{1,1}$, standard regularity estimates for minimisers imply that
$V^h_s\in H^2(S)$ with bounds
$$
\|V^h_s\|_{H^2(S)}\leq C\left( \|\mbox{div } \pa \|_{L^2(\tilde \omega)}+\|\pa\|_{L^2(\tilde \omega)}  \right).
$$
Hence Lemma \ref{lm:100} (v) ensures that $(V_s^h)$ is uniformly bounded in $H^2(S)$.
Since $V^h \to V$ in $H^1$, the bound \eqref{kor-b} therefore implies \eqref{kor-c}.

\end{proof}

\begin{lemma}
\label{tay}
There exist maps $F^h_s$, $F^h\in L^2(S; \R^{3\times 3})$ with
$$
\limsup_{h\to 0} h^{-2}\left( \|F_s^h\|_{L^2(S)} + \|F^h\|_{L^2(S)} \right) < \infty,
$$
such that
\begin{equation}
\label{tay-1}
R^h = I + h\O_{V^h} + F^h
\end{equation}
and
\begin{equation}
\label{tay-2}
R^h = I + h\O_{V^h_s} + F_s^h
\end{equation}
\end{lemma}
\begin{proof}
For brevity, we set $\mu_s^h = \mu_{V_s^h}$ and $\mu^h = \mu_{V^h}$.
\\
We first verify the tangential component of \eqref{tay-1}. Let $\tau$ be a $C^1$ tangent vector field along $S$. Then
by the definition of $V^h$ and using $\int_{I_h} R^h(x)t\AA(x)\ dt = 0$, we see that
$(I + h\nabla_{tan}V^h\ T_S)\tau$ equals
\begin{align*}
\dd_{\tau}(id + hV^h) &= \dd_{\tau} \left(\int_I y^h(x + tn(x))\ dt\right)
= \frac{1}{h}\dd_{\tau} \left(\int_{I_h} u^h(x + tn(x))\ dt\right)
\\
&= \frac{1}{h} \int_{I_h} \nabla u^h(x + tn(x)) (I + t\AA(x))\ dt \tau(x)
\\
&= R^h(x)\tau(x) - M^h(x)\tau(x),
\end{align*}
where we have introduced
$$
M^h(x) = -\frac{1}{h} \int_{I_h} \left(\nabla u^h(x + tn(x)) - R^h(x)\right)(I + t\AA(x))\ dt.
$$
Clearly,
\begin{align*}
\int_S |M^h|^2\ d\HH^2 &= \frac{1}{h}\int_{S\times I_h} |\nabla u^h(x + tn(x)) - R^h(x)|^2\ |I + t\AA(x)|^2\ d\HH^2(x)\ dt
\\
&\leq \frac{C}{h}\int_{S^h} |\nabla u^h(x) - R^h(x)|^2\ dx \leq Ch^4.
\end{align*}
To verify the normal component of \eqref{tay-1},

we compute using the tangential part of \eqref{tay-1}:
\begin{align*}
\tau\cdot R^h n &= - n\cdot R^h\tau + 2n\cdot (\sym R^h)\tau
\\
&= - n\cdot (R^h - I)\tau + 2n\cdot\sym\left( R^h - I \right)\tau
\\
&= - hn\cdot\dd_{\tau} V^h - n\cdot F^h\tau + 2n\cdot\sym\left( R^h - I \right)\tau
\\
&= h\mu^h\cdot\tau - n\cdot F^h\tau + 2n\cdot\sym\left( R^h - I \right)\tau.
\end{align*}
In the last step we used \eqref{idp-1}. As
\begin{equation}
\label{tay-7}
\left\|\sym\left( R^h - I \right) \right\|_{L^2(S)} \leq Ch^2
\end{equation}
by Lemma \ref{lm:100}, we conclude that
\begin{equation}
\label{tay-8}
\left\|T_S R^hn - h\mu^h\right\| \leq Ch^2.
\end{equation}
But again by \eqref{tay-7} we have
$$
\|n\cdot (R^h n - n)\|_{L^2(S)} = \left\|n\cdot\sym\left( R^h - I \right) n\right\|_{L^2(S)} \leq Ch^2.
$$
Since $R^h n = (n\cdot R^h n)\ n + T_S R^h n$,
we conclude that $R^h n$ agrees -- up to an error term whose $L^2(S)$-norm is dominated by $h^2$ -- with $n + h\mu^h$. This concludes
the proof of \eqref{tay-1}.

The tangential component of \eqref{tay-1} together with \eqref{kor-a} imply that the tangential component of \eqref{tay-2} is satisfied.
But then the normal component of \eqref{tay-2} follows from its tangential component in exactly the same way in which
the normal component of \eqref{tay-1} followed from its tangential component.
\end{proof}

\subsection{Two-scale convergence}

Recall that we extend the chart $r$ trivially from $S$ to $S^1$.

\begin{definition}[two-scale convergence]\label{def:two-scale}
We say that a sequence $g^h\in L^2(S^{1})$, weakly two-scale
converges in $L^2$ to the function $g\in
L^2(S^{1},L^2(\calY))$ as $h\to 0$, if the sequence $g^h$ is
bounded in $L^2(S^{1})$ and
 \begin{equation*}
   \lim\limits_{h\to 0}\int_{S^{1}} g^h(x)\,\psi(x, r(x)/\e)\ud
   x=\iint_{S^{1}\times \calY}g(x,y)\,\psi(x,y)\ud y\ud x
 \end{equation*}
 for all $\psi\in C_c^\infty(S^{1},C(\calY))$. We say that $g^h$
 strongly two-scale converges to $g$ if, in addition,
 \begin{equation*}
   \lim\limits_{h\to
     0}\|g^h\|_{L^2(S^{1})}=\|g\|_{L^2(S^{1}\times Y)}.
 \end{equation*}
\end{definition}

We write
$g^h\wtto g$ in $L^2$ (resp. $g^h\stto g$ in $L^2$) for weak
(resp. strong) two-scale convergence in $L^2$.

For the basic properties of two-scale convergence
we refer to \cite{Nguetseng-89,
   Allaire-92, Visintin-06}.
If $g^h \wtto g$ then $g^h\rightharpoonup \int_{\calY} g(\cdot,y)dy$ weakly in $L^2$.
If $g^h$ is bounded in $L^2(S^{1})$ then it has
subsequence which weakly two scale converges to some $g \in
L^2(S^{1};L^2(\calY))$.

The following lemma summarizes standard results about two scale convergence and adapts them to a possibly curved surface.
Its proof follows easily from the analogous statements for the planar case (see v) and vi) of Lemma \ref{L:two-scale} in the Appendix).

\begin{lemma} \label{lm:iggg111}
 \begin{enumerate}[(i)]
 \item if
$(g^h)_{h>0}\subset H^1(S^{1})$ is bounded, then
there exist $g_0 \in H^1(S^{1})$ and
$g_1 \in L^2(S^{1};\dot{H}^1(\calY))$ such that, after passing to a subsequence,
$\nabla g^h \wtto g$, where
$$
g=\nabla
g_0+ \nabla_y g_1(x,y).
$$
\item  if $(g^h)\subset H^2(S^{1}; \R^3)$ is bounded, then there exist  $ g_0 \in H^2(S^{1})$ and
$ g_1 \in L^2(S^{1};\dot{H}^2(\calY))$ such that, after passing to a subsequence,
$$
\nabla^2  g^h \wtto  g,
$$
where
$$
g=\nabla^2
 g_0+\sum_{i,j \leq 2} \big(\partial^2_{y_iy_j} g_{1}(x,y)\big) \taub^i \otimes \taub^j.
$$
\item  if $(g^h)$ is bounded in $H^1(S^{1}; \R^3)$ then we have $\nabla \bs g^h
\wtto \bs g$, along a subsequence,  and there exist $\bs g_0 \in H^1(S^{1};\R^3)$ and
$\bs g_1 \in L^2(S^{1};\dot{H}^1(\calY;\R^3))$ such that $$\bs g=\nabla
\bs g_0+\sum_{i \leq 3;j \leq 2} \big(\nabla_y \bs g_1(x,y)\big)_{ij} \taub^i \otimes \taub^j.$$
\item if $(g^h)$ is bounded in $H^2(S^{1}; \R^3)$ then we have $\nabla^2 \bs g^h
\wtto \bs g$, along a subsequence,  and there exist $\bs g_0 \in H^2(S^{1};\R^3)$ and
$\bs g_1 \in L^2(S^{1};\dot{H}^2(\calY;\R^3))$ such that $$\bs g=\nabla^2
\bs g_0+\sum_{i \leq 3;j,k \leq 2} \big(\partial^2_{y_jy_k} \bs g_{1,i}(x,y)\big) \taub^i \otimes \taub^j \otimes \taub^k$$
\end{enumerate}
\end{lemma}

\section{Two-scale compactness and lower bound} \label{Compactness}

Next we will identify the space of possible two scale limits of
symmetrized gradients. The following auxiliary result is standard  and it can be
easily derived, e.g., using Fourier transforms.

\begin{lemma} \label{lm:1}
Let $\B \in L^2(\omega;L^2(\calY;\R^{2 \times 2}_{\sym}))$ have the following property:
for every
$$
\Psib \in C_0^{\infty}(\omega; C^{\infty}(\calY;\R^{2 \times 2}_{\sym}))
$$
satisfying
\begin{equation}
\Psib(\xi,y)=\psi(\xi) \cof \nabla^2 F(y),
\end{equation}
for some $\psi \in C_c^\infty(\omega)$, $F\in
C^\infty(\calY)$ such that
$
\int_{\calY}  F(y) \ud y=0,
$

we have
$$
\iint_{\omega \times \calY} \B(\xi,y) :\Psib(x,y) \ud y \ud \xi
= 0.
$$
Then there exist unique $\B \in L^2(\omega;\R^{3 \times 3}_{\sym})$ and
$\wa \in L^2(\omega;\dot{H}^1(\calY;\R^2))$ such that
$$
\B(\xi,y)=\B(\xi)+\sym \nabla_y \wa(\xi,y)
$$
\end{lemma}

In what follows, we will use the notation $\otto$ introduced in the appendix. We prove the next proposition in local
coordinates; of course, this is equivalent to performing computations on the level of the surface.

\begin{proposition} \label{tm:1}

Let $(\wa^h)$ be a bounded sequence in $H^2(S;\R^3)$ such that
$$
\tfrac{1}{h}q_{\wa^h}
$$
is bounded in $L^2(S; \mathbb{S})$. Then there exist
$w_0 \in H^2(S)$, $\wa_{1} \in L^2(S;\dot{H}^2(\mathcal{Y};\R^3))$ and $\B \in \dot{L}^2(S \times Y;\mathbb{S})$ such that,
after passing to a subsequence,
\begin{equation} \label{nnnaknadno}
\nabla_{\ttan} \nabla_{\ttan} \wa^h \wtto \nabla_{\ttan} \nabla_{\ttan} \wa_0 +\sum_{i \leq 3,j,k \leq 2} \big(\partial^2_{y_jy_k} \wa_{1,i}(x,y)\big) \taub^i \otimes \taub^j \otimes \taub^k
\end{equation}
and
$$
\tfrac{1}{h}q_{\wa^h} \wtto \B.
$$
Set $\B_w =\int_Y \B(\cdot,y)\ud y$. Then the following are true:
\begin{enumerate}[(i)]
\item If $h\gg\e^2$ then there exists a unique
$v \in L^2(S;\dot{H}^1(\mathcal{Y};\R^2))$ such that
\begin{equation*} \label{eq:9}
\B = \B_w + \sum_{i,j=1,2} \left(\sym\nabla_y v\right)_{ij} \taub^i \otimes \taub^j.
\end{equation*}
\item If $h\sim\e^2$ and if we set $\lim_{h \to 0} \tfrac{\eh^2}{h}=\tfrac{1}{\gamma_1}$, then there exists a unique
$v \in L^2(S;\dot{H}^1(\mathcal{Y};\R^2))$ such that
$$
B = B_w +  \sum_{i,j=1,2} \left(\sym\nabla_y v\right)_{ij} \taub^i \otimes \taub^j +
\frac{1}{\gamma_1} \wa_{1,3} \AA
$$
\item If $h \ll \e^2$, then there exists a unique
$v \in L^2(S;\dot{H}^1(\mathcal{Y};\R^2))$ such that
$$
\sum_{i,j=1,2} \left(\sym\nabla_y v\right)_{ij} \taub^i \otimes \taub^j + (\wa_{1,3}) \AA = 0.
$$
\end{enumerate}

\end{proposition}

\begin{proof}
The existence of $w_0$ and $w_1$ follows from Lemma \ref{lm:iggg111}.
In local coordinates (\ref{nnnaknadno}) can be expressed as
$$ \nabla^2 \tilde{w}^h \wtto \nabla^2 \tilde{w}_0+\nabla^2_y \check{w}_1,$$
where $\tilde{w}^h=w^h\circ \psi$, $\tilde{w}_0=w_0 \circ \psi$, $\check{w}_{1}=\tilde{w}_{1,i}  \tilde{\taub}^i$, $\tilde{w}_{1}= w_{1}\circ \psi$, $\tilde{\taub}^i=\taub^i \circ \psi$ where by slight abuse of notation we write $(w_{1,i}\circ \psi)(x,y)=w_{1,i}(\psi(x),y)$, for $(x,y) \in \omega \times \calY$. Denote also $\bar{w}^h_\alpha=\tilde{w}^h \cdot \partial_\alpha \psi$, $\bar{w}_{0,\alpha}=w_0 \cdot \partial_\alpha \psi$, for $\alpha=1,2$ and $\bar{w}_{1}=(\check{w}_{1,1},\check{w}_{1,2})$.
By Lemma \ref{L:two-scale-h1} in the appendix we have
\begin{equation}
\label{merd-2}
\frac{1}{\e} \sym \nabla \bar{w}^h \otto  \sym \nabla_y \bar{w}_1, \quad \frac{1}{\e} \sym \left((\nabla \psi)^T \nabla\o{w}^h_{\ttan} \right) \otto  \sym \nabla_y \bar{w}_1.
\end{equation}
and
\begin{equation}
\label{merd-3}
\frac{\tilde{w}^h}{\e^2} \otto \check{w}_1, \quad \frac{\bar{w}^h}{\e^2} \otto \bar{w}_1, \quad \frac{\tilde{w}^h\cdot n}{\e^2} \otto \tilde{w}_{1,3}.
\end{equation}

Now let $F\in C^{\infty}(\mathcal{Y})$ and $\p\in C^{\infty}_0(\omega)$.
With \eqref{merd-1} in mind, we compute:
\begin{align*}
&\int_{\omega} \sym  \nabla\bar{w}^h  : (\cof \nabla^2 F)\left( \frac{\cdot}{\e} \right) \p
\\
&= \e^2\int_{\omega}   \sym  \nabla\bar{w}^h  : \cof
\left[
\nabla^2 \left( F\left( \frac{\cdot}{\e} \right) \p\right)
- \frac{2}{\e} (\nabla F)\left( \frac{\cdot}{\e} \right)\otimes \nabla\p - F\left( \frac{\cdot}{\e}\right) \nabla^2\p
\right]
\\
&= -\e^2\int_{\omega}  \sym  \nabla\bar{w}^h: \cof
\left[\frac{2}{\e} (\nabla F)\left( \frac{\cdot}{\e} \right)\otimes \nabla\p + F\left( \frac{\cdot}{\e}\right) \nabla^2\p
\right].
\end{align*}

We used that the term  $\sym \nabla\o{w}^h $ is
$L^2$-orthogonal to test matrix fields of the form $\cof \nabla^2 F$. From this and from \eqref{merd-1} and \eqref{merd-3} we deduce that
\begin{align}\label{jedandan1}
 &\int_{\omega} \frac{\sym \left((\nabla \psi)^T \nabla\o{w}^h_{\ttan} \right)}{\e^2} :(\cof \nabla^2 F) \left(\frac{\cdot}{\e} \right) \p
\\ \nonumber
&= - \int_{\omega} \frac{\sym \nabla\bar{w}^h}{\e} : \cof
\left[2 (\nabla F)\left( \frac{\cdot}{\e} \right)\otimes \nabla\p + \e F\left( \frac{\cdot}{\e}\right) \nabla^2\p
\right]
- \int_{\omega} \frac{\Gamma\cdot\bar{w}^h}{\e^2} : (\cof \nabla^2 F)\left( \frac{\cdot}{\e} \right) \p
\\ \nonumber
&\to - \int_{{\omega}\times Y}
\left[
{\sym_y \nabla \bar{w}_1}(\cdot, y) : \cof
\left[2 (\nabla F)\left( y \right)\otimes \nabla\p \right]
+
(\Gamma\cdot\bar{w}_1(\cdot, y)) : (\cof \nabla^2 F)\left( y \right) \p
\right].
\end{align}
Recall the identity
\begin{equation} \label{jedandan}
q_{w^h}\equiv \sym \left((\nabla \psi)^T \nabla\o{w}^h \right)= \sym \left((\nabla \psi)^T \nabla\o{w}^h_{\ttan} \right)+(\tilde{w}^h \cdot n) \AA.
\end{equation}
Assume first that $h \gg \e^2$. From the assumption $\frac{1}{h} q_{w^h} \wtto B$ and (\ref{merd-3}), (\ref{jedandan1}) and Lemma \ref{lm:1}
we conclude that there exists $\check{v} \in L^2(\omega;\dot{H}^1(\calY;\R^3))$ such that
$$ \sym \left((\nabla \psi)^T \nabla\o{w}^h \right) \wtto \tilde{B}_w+\sym \nabla_y \check{v}. $$
Here $\tilde{B}_w=B_w \circ \psi$. This is i) in local coordinates, after defining $v_i=(\check{v} \circ \psi^{-1})\cdot \taub_i$ for $i=1,2,3$.
Case ii) we conclude as follows: By dividing the identity (\ref{jedandan}) by $\e$ and using
$$
\frac{q^h}{\e} \to 0 \mbox{ strongly in }L^2,\quad \text{when } h \ll \e,
$$
as well as (\ref{merd-2}) and (\ref{merd-3}) we conclude that $\bar{w}_1=0$ from Korn's inequality.
Using (\ref{jedandan1}) and the identity (\ref{jedandan}), after dividing it by $\e^2$, we conclude
\begin{align}
\nonumber
\lim_{h\to 0}\int_{\omega} \frac{\sym \left((\nabla \psi)^T \nabla\o{w}^h \right) }{\e^2} : (\cof \nabla^2 F)\left( \frac{\cdot}{\e} \right) \p
&= \lim_{h\to 0} \int_{\omega} (\frac{\tilde{w}^h}{\e^2}\cdot n)\AA : (\cof\nabla^2 F)\left( \frac{\cdot}{\e} \right)\p
\\
\label{merd-5}
&= \int_{{\omega}\times Y} \tilde{w}_{1,3}(\cdot, y)\AA: (\cof \nabla^2F)(y)\p.
\end{align}
Lemma \ref{lm:1} again shows that there exists $v$ such that
$$
\frac{q^h}{h} \wtto B_w + \sum_{i,j=1}^2(\sym\nabla_y v)_{ij} \taub^i \otimes \taub^j + \left( \lim_{h\to 0} \frac{\e^2}{h} \right) w_{1,3} \AA.
$$

For the case (iii) we argue similarly: as in the case of ii) we conclude $\bar{w}_1=0$. Also we know that
$
{q^h}/{\e^2} \to 0
$
strongly in $L^2$. Hence the left-hand side of \eqref{merd-5} converges to zero, so
$$
 \int_{{\omega}\times Y} \tilde{w}_{1,3} (\cdot, y) \AA: (\cof \nabla^2F)(y)\p = 0,
$$
which by Lemma \ref{lm:1} implies the claim.
\end{proof}

\begin{lemma}
\label{lm:50}
Let $(\wa^h)_{h>0} \subset H^1(S^{1};\R^3)$ be such that
    \begin{equation*}
      \limsup\limits_{h\to 0}\left( \,\|\wa^h\|_{L^2}+\|\nabla_h\wa^h\|_{L^2}\,\right)<\infty.
    \end{equation*}
    Then there exists a map $\wa_0 \in H^1(S;\R^3)$ and a field $\bs H_\gamma \in L(S \times I \times \calY; \R^{3 \times 3})$ such that
  \begin{equation}\label{P:comp1}
    \bs H_\gamma=
    \left\{\begin{aligned}
        &\sum_{i,j=1}^3\left(\,\nabla_y\wa_1,\wa_2 \right)_{ij} \taub^i \otimes \taub^j
        &&\text{for some }
        \left\{\begin{aligned}
            &\wa_1\in L^2(S;\dot{H}^1(\mathcal Y;\R^3))\\
            &\wa_2\in L^2(S \times Y \times I; \R^3))
          \end{aligned}\right\}\\
        &&&\text{if $\gamma=0$},\\
        &\sum_{i,j=1}^3\left(\,\nabla_y\wa_1,\tfrac{1}{\gamma}\partial_3\wa_1\,\right)_{ij} \taub^i \otimes \taub^j
        &&\text{for some $\wa_1\in L^2(S ,\dot{H}^1(I\times\mathcal Y; \R^3))$}\\
        &&&\text{if $\gamma\in(0,\infty)$},\\
        &\sum_{i,j=1}^3 \left(\,\nabla_y\wa_1,\wa_2\,\right)_{ij} \taub^i \otimes \taub^j
        &&\text{for some }
        \left\{\begin{aligned}
            &\wa_1\in L^2(S \times I; \dot{H}^1(\mathcal Y; \R^3))\\
            &\wa_2\in L^2(S \times I;\R^3))
          \end{aligned}\right\}\\
        &&&\text{if $\gamma=\infty$},\\
      \end{aligned}\right.
  \end{equation}
  such that, up to a
    subsequence, $\wa^h\to\wa_0$ in $L^2$ and
    \begin{equation*}
      \nabla_h\wa^h\wtto \nabla_{\ttan}\wa_0\, T_S +\bs H_\gamma \qquad\text{weakly
        two-scale in }L^2.
    \end{equation*}
Here, $\wa_0$ is the weak limit in $H^1(S)$ of
$\int_I \wa^h(x+tn(x))dt$.
\end{lemma}

\begin{proof}

The lemma is an analogue of Proposition~6.3.5 in \cite{Neukamm-10}, adapted to  the manifold $S$ and the definition of two scale convergence on the manifold. Thus we will only prove the case $\gamma \in  (0,\infty)$.
Since the sequence is bounded in $H^1$ norm there exists a weak
limit $\wa_0 \in H^1¸(S^{1};\R^3)$. Let us denote by
$\tilde{\wa}^h,\ \tilde{\wa}_0$  the elements of
$H^1(\Omega;\R^3)$,
 defined by:
\begin{equation}
\tilde{\wa}^h=\wa^h \circ \ra_e^{-1}, \quad \tilde{\wa}_0=\wa_0 \circ \ra_e^{-1}.
\end{equation}
By Proposition~6.3.5 in \cite{Neukamm-10}, which is proved for planar domains, $\tilde{\wa}_0$ does not depend on $t$
and there exists $\check{\wa}_1 \in L^2(\omega;\dot{H}^1(I \times
\calY;\R^3))$ such that:
\begin{equation} \label{eq:2000}
\nabla_h \tilde{\wa}^h \wtto
(\nabla_\xi \tilde{\wa}_0,0)+( \nabla_y
\check{\wa}_1,\tfrac{1}{\gamma}
\partial_t \check{\wa}_1).
\end{equation}
Then we have that $\breve{\wa}_1 \in L^2(S;\dot{H}^1(I
\times\calY;\R^3))$. Using (\ref{eq:400}) and (\ref{fladh}) we
conclude
\begin{eqnarray}
\nabla_h \wa^h &=& \nabla \wa^h (\I+t\AA)(\I_h+th\AA)^{-1}
\\
\nonumber&=&  \nabla \tilde{\wa}^h[\taub^1,\taub^2,\taub^3]^T (\I+t\AA)^{-1}(\I+t\AA)(\I_h+th\AA)^{-1} \\
\nonumber&=& \nabla_h \tilde{\wa}^h [\taub^1,\taub^2,\taub^3]^T (\I+th\AA)^{-1}.
\end{eqnarray}
By using (\ref{eq:2000}) we conclude that
\begin{eqnarray}
\nabla_h \wa^h &\wtto& \left((\nabla_\xi \tilde{\wa}_0)\circ \ra,0 \right) [\taub^1,\taub^2,\taub^3]^T+( \nabla_y
\breve{\wa}_1\circ \ra_e,\tfrac{1}{\gamma}
\partial_t \breve{\wa}_1\circ \ra_e ) [\taub^1,\taub^2,\taub^3]^T \\
\nonumber &=&  \nabla \wa_0 \, T_S+\sum_{i,j=1}^3 ( \nabla_y \wa_1,\tfrac{1}{\gamma} \partial_t \wa_1)_{ij} \taub^i \otimes \taub^j,
\end{eqnarray}
where $(\wa_1)_i= (\breve{\wa}_1 \circ \ra_e) \cdot \taub_i $. The
last property follows from the fact that $ \tilde{\wa}_0$ does
not depend on $t$.
\end{proof}

The following lemma is fairly straightforward; we refer to \cite[Corollary 2.3.4]{Neukamm-10} for a proof.

\begin{lemma} \label{lm:10}
Let $(\bs E^h_{\text{app}})\subset L^2(\Omega; \R^{3\times 3})$ be such that
$$
\bs E^h_{\text{app}} \stackrel{2, \gamma}{\weak}  \bs E_{\text{app}} \textrm{ in } L^2(\Omega \times \calY; \R^{3\times 3}).
$$
Then
$$
h^{-2}\left(\sqrt{(\I+h^2 \bs E^h_{\text{app}})^T(\I+h^2 \bs E^h_{\text{app}})}-\I\right)
\stackrel{2, \gamma}{\weak} \sym \bs E_{\text{app}} \ \textrm{ in }  L^2(\Omega \times \calY;\R^{3\times 3})
$$
\end{lemma}

\begin{proposition} \label{tm:ggggg}
Assume that there is $B \in L^2(S \times \calY;\QQ)$ such that
$$
\tfrac{1}{h}q_{\V^h_s} \wtto B \mbox{ in }L^2(S; \QQ),
$$
and assume that there is ${B}_w \in L^2(S;\QQ)$ such that
$$
\tfrac{1}{h}q_{\V^h} \rightharpoonup {B}_{w}, \text{ weakly in }L^2(S; \QQ).
$$
Assume also that
$$
\nabla_{\ttan} \nabla_{\ttan}(V^h_s \cdot n) \wtto \nabla_{\ttan} \nabla_{\ttan} (V \cdot n)+\sum_{i,j=1,2} (\partial^2_{y_i y_j} \varphi) \taub^i \otimes \taub^j,
$$

and set
\begin{equation} \label{p36-1444}
\bs E^h =\frac{\sqrt{(\nabla_h \ya^h)^T \nabla_h \ya^h}-\I}{h^2}.
\end{equation}

Then there exist
$$
\bs U_\gamma \in  L^2(S; L_\gamma(I\times\mathcal Y))
$$
such that (after passing to a subsequence)
$$
\bs E^h \wtto \bs E\mbox{ in }L^2(S^1; \R^{3\times 3}),
$$
where $\bs E$ is given by
\begin{equation}
\label{p36-1}
\bs E = \left(B_w+\dot{B} + \tfrac{1}{2} (dV)^2  - t \bbv\right)(T_S, T_S)
- t\sum_{i,j=1}^2 (\partial^2_{y_i y_j} \varphi)\taub^i \otimes \taub^j +\bs U_{\gamma},
\end{equation}
and where $\dot{B}=B-\int_Y B(\cdot,y) \ud y$.

In particular, the following are true:
\begin{enumerate}[(i)]
\item If $\gamma \in ( 0,\infty ]$ then there exists $\bs U_\gamma \in L^2(S;L_\gamma(I\times\mathcal Y))$ such that
\begin{equation} \label{eq:iggg556}
\bs E=\B_w + \tfrac{1}{2} (dV)^2 - t \bbv + {\bs U}_{\gamma}.
\end{equation}
\item If $\e^2\ll h \ll \e$, there exists $\bs U \in L^2(S;L_0^0(I\times\mathcal Y))$ such that

 \begin{equation} \label{p36-16}
\bs E=\B_w + \tfrac{1}{2} (dV)^2 - t \bbv +\bs U.
\end{equation}
\item If $h\sim \e^2$, with $\lim_{h \to 0} \tfrac{\e^2}{h}=\tfrac{1}{\gamma_1} \in ( 0,\infty)$, there exists $\bs U \in L^2(S;L_{0,\gamma_1}^{1}(I\times\mathcal Y))$
 such that
\begin{equation} \label{p36-156}
\bs E = \B_w + \tfrac{1}{2} (dV)^2 - t \bbv +\bs U.
\end{equation}

\end{enumerate}
\end{proposition}
\begin{proof}
By Lemma \ref{lm:iggg111}, there exists a subsequence such that $E^h \wtto E$ for some $E$.
Denote by $\bs E^h_{\text{app}}$ the approximate strain
\begin{equation}
\bs E^h_{\text{app}} :=\frac{(\Ra^h)^T \nabla_h \ya^h -\I}{h^2}
\end{equation}
By Lemma \ref{lm:10} it is enough to identify the two-scale limit of
$\sym \bs E^h_{\text{app}}$.
Let us write
\begin{equation} \label{eq:249}
\Ra^h \bs E^h_{\text{app}}=\frac{\nabla_h \ya^h-\I}{h^2}-\frac{\Ra^h-\I}{h^2}.
\end{equation}
We have $\sym (\Ra^h \bs E^h_{\text{app}}) \wtto
\bs E$, because $\Ra^h \to \I$ boundedly in measure. By property (vi) of Lemma \ref{lm:100} the symmetric
part of the second term converges strongly in $L^2$
(and thus two-scale) to $\O_V^2/2$. So we need to identify the
two scale limit of
$$
\sym \left(\frac{\nabla_h \ya^h-\I}{h^2}\right).
$$
\\
For brevity we set $\mu_s^h = \mu_{V_s^h}$. As usual, we extend $V_s^h$, $n$ and $\mu_s^h$ trivially to $S^h$.

In what follows we abuse notation using $t$ also as an independent variable.
We define the maps $z^h : S^h\to\R^3$ by setting
$$
z^h(x) = x + h \left( V_s^h(x) + t(x)\mu_s^h(x) \right)\mbox{ for all }x\in S^h.
$$

Define $Q(x)$ as in \eqref{defQ} and define (compare \eqref{linwein})
$$
b^h(x) = b_{V^h_s}(x) \equiv - \nabla_{tan}\mu_s^h(x)T_S(x) + \nabla_{tan}V_s^h(x)\AA(x)
$$
and (compare \eqref{defO})
$$
\O^h(x) = \O_{V^h_s}(x) \equiv \nabla_{tan}V^h_s(x) T_S(x) + \mu^h_s(x)\otimes n(x).
$$
Then Lemma \ref{ansatz} shows that
\begin{equation}
\label{Dz-1}
\nabla z^h = I + h\O^h - ht\, b^h
- ht^2 \nabla_{tan}\mu_s^h \AA + h\left( \nabla_{tan}V_s^h + t\nabla_{tan}\mu_s^h \right)T_S Q.
\end{equation}
Note that $|Q| \leq Ct^2 \leq Ch^2$ on $S^h$, so
$$
\|Q\|_{L^2(S^h)} \leq Ch^{5/2}.
$$
In what follows $\Theta^h\in L^2(S^h)$ denote maps which may change from expression to expression, but which always satisfy
$$
\|\Theta^h\|_{L^2(S^h)} \leq Ch^{5/2}.
$$
We see from \eqref{Dz-1} that
$$
\nabla z^h = I + h\O^h + \Theta^h = R^h + \Theta^h,
$$
by \eqref{tay-2}. On the other hand, Lemma \ref{lm:100} shows that
$
\nabla u^h = R^h + \Theta^h.
$
Hence
\begin{equation}
\label{Dz-2}
\|\nabla u^h - \nabla z^h\|_{L^2(S^h)} \leq Ch^{5/2}.
\end{equation}
However, by the definition of $V^h$ and of $z^h$ we have, for $x\in S$,
$$
\frac{1}{h}\int_{I^h} z^h(x + tn(x)) - u^h(x + tn(x))\ dt = h\left( V^h_s(x) - V^h(x) \right).
$$
Hence Poincar\'e's inequality implies that
\begin{align*}
\|u^h - z^h\|_{L^2(S^h)} &\leq \left\|u^h - z^h - h\left( V^h_s - V^h \right)\right\|_{L^2(S^h)} + h\|V^h_s - V^h\|_{L^2(S^h)}
\\
&\leq \|\nabla u^h - \nabla z^h\|_{L^2(S^h)} + h^{3/2}\|V^h_s - V^h\|_{L^2(S)}
\leq Ch^{5/2},
\end{align*}
by \eqref{Dz-2} and \eqref{kor-b}.
Defining $Z^h : S^1\to\R^3$ by setting
$
Z^h(\Phi^h) = z^h
$
on $S^h$,
we have the equivalent bounds
$$
\|y^h - Z^h\|_{L^2(S^1)} + \|\nabla_h (y^h - Z^h )\|_{L^2(S^1)}\leq Ch^2.
$$

Thus, using Lemma \ref{lm:50}, we conclude that there
exists
$$
\bs H_\gamma \in L^2(S^1 \times \calY;R^{3\times 3})
$$
of the form given in Lemma \ref{lm:50} and $c \in H^1(S;\R^3)$ such that (after passing to a subsequence)
\begin{equation} \label{eq:112}
\frac{1}{h^2}\nabla_h \left(\ya^h- Z^h\right) \wtto \nabla_{tan}c\, T_S+ \bs H_\gamma.
\end{equation}
Here $c$ is a weak limit in $H^1(S;\R^3)$ of $\tfrac{\V^h-\V^h_s}{h}$.

We will now identify the two scale limit on $S^1$ of the quantity $\sym\left(\frac{\nabla_h Z^h-\I}{h^2}\right)$.
By \eqref{Dz-1} we have for all $x\in S^h$:
\begin{equation}
\label{Dz-5}
\frac{\nabla  z^h(x) - I}{h^2} = \frac{1}{h}\, \O^h(x) - \frac{t(x)}{h}\, b^h(x) + M^h(x),
\end{equation}
where
$$
\|M^h\|_{L^2(S^h)} \leq Ch^{3/2}.
$$
We must therefore identify the two-scale limits on $S$ of the first two terms.
\\
Lemma \ref{irre} implies
\begin{equation}
\label{symPh}
\frac{1}{h}\sym \O^h = \frac{1}{h}q_{V^h_s}(T_S, T_S) \wtto B(T_S, T_S)
\end{equation}
weakly two-scale in $L^2(S; \R^{3\times 3})$, by definition of $B$.
\\
It remains to identify the two-scale limit of $b^h$ on $S$.
Its weak $L^2$-limit clearly is $b_V$. This follows by comparing the definition of $b^h$
with \eqref{linwein}.
Next note that, since $V_s^h\to V$ strongly in $H^1(S)$, we know that
$
\nabla_{tan} V_s^h\AA
$
does not contribute to the oscillating part.
On the other hand, 
we have from \eqref{idp-1}:
\begin{align*}
\nabla_{tan}\mu^h
&= - \nabla_{tan}\nabla_{tan}\left( n\cdot V_s^h \right) + \nabla_{tan}\left( V_s^h\cdot\AA \right).
\end{align*}
The last term converges strongly in $L^2(S)$ to $\nabla_{tan}\left( V\cdot\AA \right)$, so it does not contribute to the oscillating part.
The contribution of the term $\nabla_{tan}\nabla_{tan}\left( n\cdot V_s^h \right)$ is given by the assumption.

We conclude that
\begin{equation}
b^h \wtto b_V + \dd^2_{y_iy_j}\varphi\ \tau^i\otimes\tau^j
\end{equation}
on $S$.
By \eqref{Dz-5}, the above convergence results on $S$ imply that
\begin{equation} \label{eq:iggg444}
\sym\left(\frac{\nabla_h Z^h-\I}{h^2}\right) \wtto
\left(B - tb_V\right)(T_S, T_S) -t (\partial^2_{y_i y_j} \varphi) \taub^i \otimes \taub^j
\end{equation}
weakly two-scale on $S^1$.

We conclude from (\ref{eq:112}) and (\ref{eq:iggg444}) that
$$
\bs E =\B_{} + \sym\big(\nabla_{\ttan} c \, T_S \big)
-\tfrac{1}{2} (\O_V^2)-t\sum_{i,j=1}^2 (\partial^2_{y_i y_j} \varphi)\taub^i \otimes \taub^j - t\bbv +\tilde{\bs U}_{\gamma}
$$
for some
$ \tilde{\bs U}_\gamma \in  L^2(S; L_\gamma(I\times\mathcal Y)).$
 Notice that
\begin{equation}
 B + \sym\big(\nabla_{\ttan}  c \, T_S\big)
 = B_w + \dot{B} + \bar{U}_\gamma, \end{equation}
 where
 $$\bar{U}_\gamma:= \tfrac{1}{2}\sum_{i=1}^2 (\nabla c \taub_i, \na)(\na \otimes \taub^i+\taub^i \otimes \na).$$
Notice that $ \bar{U}_\gamma \in L^2(S;L_\gamma(I\times \calY))$.

Define
$$
U_\gamma = \tilde{\bs U}_{\gamma} +\bar{\bs U}_{\gamma}  \in L^2(S;L_\gamma(I\times \calY)).
$$
Hence (\ref{p36-1}) is proven, using the identity (\ref{form-2}).
The remaining claims
now follow from Proposition \ref{tm:1}.
Namely, ii) and iii) are direct and i) is the consequence of the identity
\begin{equation*}
\left(\begin{array}{cc} t (\partial_{y_i y_j} \varphi)_{i,j=1}^2 & 0 \\ 0& 0 \end{array} \right)=\sym (\nabla_y, \tfrac{1}{\gamma} \partial_3) \left(\begin{array}{c} t \partial_{y_1} \varphi \\ t \partial_{y_2} \varphi  \\ \tfrac{1}{\gamma} \varphi \end{array} \right) \in L_{\gamma}(I\times \calY),
\end{equation*}
for $\gamma \in (0,\infty)$.

\end{proof}

For lower bound we need the following lemma.
\begin{lemma} \label{lm:iggg11111}
Let $(\ya^h) \subset H^1(S^{1};\R^3)$, define $E^h : S^1\to\R^{3\times 3}$ by
$$
\sqrt{(\nabla_h \ya^h)^T (\nabla_h \ya^h)} = I+h^2 \bs E^h,
$$
and assume that $\bs E^h \wtto \bs: E$. Then we have
\begin{eqnarray*}
& &\liminf_{h \to 0}  \int_{S}\int_I \mathcal{Q} (x+tn(x), r(x+tn(x))/\e, \bs E^h(x+tn(x)))\ud t \ud \mathcal{H}^2 \geq\\
& &\int_{S} \int_I \int_{\calY} \mathcal{Q} (x+tn(x),y, \bs E(\hx+tn(x),y)) \ud y \ud t \ud \mathcal{H}^2
\end{eqnarray*}
and
\begin{eqnarray*}
& &\liminf_{h \to 0}  \frac{1}{h^4}\int_{S}\int_I  W(x+tn(x), r(x+tn(x))/\e, \bs I+h^2\bs E^h(x+tn(x)))\, \ud t \ud  \mathcal{H}^2 \geq \\
& &\int_{S} \int_I \int_{\calY} \mathcal{Q} (x+tn(x),y, \bs E(x+tn(x),y)) \ud y \ud t \,d\mathcal{H}^2.
\end{eqnarray*}
\end{lemma}

\begin{proof}
For the first claim we refer to \cite{Visintin-06,Visintin-07}. The second claim
then follows from the standard truncation argument.
\end{proof}

The lower bound parts of Theorem \ref{tm:glavni} and Theorem \ref{tm:glavni2}
is now a direct consequence of Proposition~\ref{tm:ggggg} and of Lemma~\ref{lm:iggg11111}.

\section{Upper bound} \label{peti}
We start with the following observation.
\begin{remark} \label{rem:zzzadnje} %
It is easy to see, by using Korn's inequality,  that $L_\gamma(I \times\mathcal Y,\R^{3 \times 3}_{\sym})$ as well as $L_0^0(I \times\mathcal Y,\R^{3 \times 3}_{\sym})$, $L_{0,\gamma_1}^{1}(I \times\mathcal Y,\R^{3 \times 3}_{\sym})$ for $\gamma \in [0,\infty]$ and $\gamma_1 \in (0,\infty)$ are closed subspaces of $L^2 (I \times\mathcal Y,\R^{3 \times 3}_{\sym})$.
Also by using Korn's inequality it is easy to see (see also \cite{Neukamm-11,Neukamm-10,NeuVel-12}) that the following coercivity bounds are satisfied:
\begin{eqnarray*}
\|\zetaa\|_{H^1}^2+\|\bs g\|_{L^2}^2 &\lesssim&  \|\mathcal{U}_0(\zetaa,\bs g)\|^2_{L^2},\\ && \hspace{-10ex} \forall \zetaa \in \dot{H}^1(\calY;\R^2), \bs g \in L^2(I\times \calY;\R^3), \\ && \\
\|\zetaa\|_{H^1}^2+\|\varphi\|_{H^2}+\|\bs g\|_{L^2}^2 &\lesssim&  \|\mathcal{U}_0^0(\zetaa,\varphi, \bs g)\|^2_{L^2},\\ && \hspace{-15ex} \forall \zetaa \in \dot{H}^1(\calY;\R^2),\varphi \in \dot{H}^2(\calY), \bs g \in L^2(I\times \calY;\R^3) \text{ and }  \\ && \\
\|\zetaa\|_{H^1}^2+\|\varphi\|_{H^2}+\|\bs g\|_{L^2}^2 &\lesssim&  \|\mathcal{U}_{0,\gamma_1}^{1} (\zetaa,\varphi,\bs g)\|^2_{L^2},\\ && \hspace{-15ex} \forall \zetaa \in \dot{H}^1(\calY;\R^2),\varphi \in \dot{H}^2(\calY), \bs g \in L^2(I\times \calY;\R^3) \text{ and } \gamma_1 \in (0,\infty) \\ && \\
\|\zetaa\|_{H^1}^2+\|\psi\|_{H^1}+\|\bs c\|_{L^2}^2 &\lesssim&  \|\mathcal{U}_\infty (\zetaa,\psi,\bs c)\|^2_{L^2},\\ && \hspace{-10ex} \forall \zetaa \in \dot{H}^1(\calY;\R^2),\psi \in L^2(I;\dot{H}^1(\calY)), \bs c \in L^2(I;\R^3),\\ && \\
\|\phia\|_{H^1}^2 &\leq& C(\gamma) \|\mathcal{U}_\gamma (\phia)\|_{L^2}^2, \ \forall \phia \in \dot{H}^1(I \times \calY;\R^3).
\end{eqnarray*}
Here the constant absorbed into the symbol $\lesssim$ depends on $\eta_1,\eta_2$.
\end{remark}

The following two lemmas and remark are analogous to \cite[Lemma~2.10, 2.11]{NeuVel-12}.

\begin{lemma}
  \label{lem:2a}
  For $\gamma \in (0,\infty]$ there exists a bounded linear operator
  \begin{equation*}
    { \Pi_\gamma}: L^2(S,\QQ)\times L^2(S,\QQ)\to L^2(S,L_\gamma(I\times\mathcal Y))
  \end{equation*}
  such that for almost every $x \in S$ we have
$$
 \mathcal{Q}_{\gamma}(x, q^1, q^2) =
\int_I\int_{\calY} \mathcal{Q}\Big(x + tn(x), y, q^1 + tq^2 + {\bs\Pi_\gamma}[q^1, q^2](x, t, y) \Big) \ud y \ud t.
$$
Moreover, if $q^1, q^2 \in C(S,\QQ)$ then
$$
(x,t,y)\mapsto{\bs \Pi_\gamma}[q^1, q^2](x,t,y)
$$
is continuous as well.
\end{lemma}

\begin{lemma}
  \label{lem:2b}
  The function
  $\mathcal{Q}_\gamma:S\times\QQ \times\QQ\to\R^+$ is continuous.
  Moreover for any $x \in S$ the function $Q_\gamma(x):\mathbb{S}(x)_{\sym} \times \mathbb{S}(x)_{\sym} \to \R$  satisfies
  \begin{eqnarray*}
  (|q^1|^2 + |q^2|^2) &\lesssim& \mathcal{Q}_\gamma( x,q^1,q^2)= \mathcal{Q}_\gamma( x,q^1,q^2)\\
        &\lesssim& |q^1|^2 + |q^2|^2.
      \end{eqnarray*}
The constant in $\lesssim$ depends only on $\alpha,\beta,\eta_1,\eta_2$.
\end{lemma}

Lemma \ref{lem:2a} and Lemma \ref{lem:2b} remain true
for the quadratic forms $\mathcal{Q}_0^0$, $\mathcal{Q}_0^{1,\gamma_1}$,
the spaces $L_0^0(I\times \calY)$, $L_{0,\gamma_1}^{1}(I\times \calY)$,
and the appropriate operators
$$
\Pi_0^0:L^2(S,\QQ)\times L^2(S,\QQ)\to
    L^2(S,L_0^0(I\times\mathcal Y))
$$
and
$$
\Pi_{0,\gamma_1}^{1}:L^2(S,\QQ)\times L^2(S,\QQ)\to
    L^2(S,L_{0,\gamma_1}^{1} (I\times\mathcal Y))
$$
for $\gamma_1 \in (0,\infty)$.

We need the following
auxiliary result concerning the linearization of the square root of a matrix.
Its proof is straighforward by Taylor expansion.

\begin{lemma}
  \label{lem:3}
There exists $\eta>0$ and a nondecreasing function $m: (0,\eta) \to \R^+$ such that $m (\delta) \to 0$ as $\delta\to 0$ such that the following is true:
\\
Let $G^h\in L^2(S^1, \R^{3\times 3})$ and $K^h\in L^4(S^1,\R^{3\times 3})$ satisfy the following conditions:
\begin{enumerate}[(i)]
\item There exists $M\in\R$ such that
$$
  \limsup_{h \to 0}  \left(\|\sym  \bs G^h \|_{L^2(\Omega)}+\|\bs K^h \|_{L^4(\Omega)} \right) \leq M.
$$
 \item We have
$
h^{-1}\sym \bs K^h \to 0 \mbox{ strongly in }L^2(\Omega)
$
as $h\to 0$.
\item There exists $\delta < \eta$ such that
$$\limsup_{h \to 0} h\|\bs K^h \|_{L^\infty} \leq  \delta
$$
and
$$
\limsup_{h \to 0} h^2\|\bs G^h \|_{L^\infty} \leq \delta.
$$
\item We have $hG^h \to 0$ strongly in $L^4(\Omega)$ as $h\to 0$.
\end{enumerate}
Set
  \begin{equation*}
    \bs E^h = \frac{1}{h^2}\left(\sqrt{(\id+h\bs K^h+h^2\bs G^h)^t(\id+h\bs K^h+h^2\bs G^h)} - \id\right)
  \end{equation*}
and
\begin{equation}
\bs E^h_{app} = \sym\bs G^h-\frac{1}{2}(\bs K^h)^2.
\end{equation}
Then we have
  \begin{equation} \label{opravdaj1}
    \limsup\limits_{h\to 0}\|\bs E^h-\bs E^h_{app} \|_{L^2}\leq M^2 m(\delta),
  \end{equation}
and
 \begin{align} \label{opravdaj2}
\limsup_{h \to 0} &\Big| \frac{1}{h^4}\int_{S}\int_I W(x+tn(x), r(x+tn(x))/\e, \bs I+h^2 \bs E^h(x+tn(x)))\ud t \ud \mathcal{H}^2\\ \nonumber
&-\int_{S}\int_I \mathcal{Q} (x+tn(x), r(x+tn(x))/\e, \bs E^h_{app}(x+tn(x)) \ud t \ud \mathcal{H}^2\Big| \leq (M+1)^4 m(\delta).
  \end{align}
If, moreover, $\bs E^h_{app} \stto \bs E(x,y)$ strongly two scale,
then
\begin{align} \label{opravdaj3}
\limsup_{h \to 0} &\Big| \frac{1}{h^4}\int_{S}\int_I
       W(x+tn(x), r(x+tn(x))/\e, I + h^2\bs E^h(x+tn(x)) \ud t \ud \mathcal{H}^2 \\ \nonumber
& -\int_{S}\int_I \int_{\calY} \mathcal{Q} (x+tn(x),y, \bs E(x+tn(x),y)) \ud y \ud t \ud \mathcal{H}^2 \,dx\Big| \leq (M+1)^4 m(\delta).
\end{align}
\end{lemma}
\begin{proof}
We will just give the sketch of the proof. By Taylor expansion there exists $\eta_1>0$ and nondecreasing function $m_1:(0,\eta_1) \to \R^+$ such that $m_1(\delta) \to 0$ as $\delta \to 0$ and for every
$A\in \R^{3 \times 3}$ which satisfies $|A-I|<\eta_1$ we have
 \begin{equation} \label{idd1}
 \left| \sqrt{(I+A^t)(I+A)}-\left(I+\sym A+\frac{1}{2} A^t A\right)\right|\leq m_1(|A-I|)\left( \sym A+\frac{1}{2} A^T A \right).
\end{equation}
If we plug into this identity $A=hK^h+h^2G^h$, divide by $h^2$ and integrate and let $h \to 0$ we obtain
$$ \limsup_{h \to 0} \|E^h-\tilde{E}^h_{app}
\|_{L^2} \leq m_1(\delta) \|\tilde{E}^h_{app}\|_{L^2}, $$
where
$$ \tilde{E}^h_{app}=\frac{\sym K^h}{h}+\sym G^h+\tfrac{1}{2}(K^h)^tK^h+h\sym \left((G^h)^tK^h\right)+\tfrac{1}{2} h^2(G^h)^tG^h.  $$
Using the assumptions it is easy to prove that $$ \lim_{h \to 0} \| \tilde{E}^h_{app}-E^h_{app}\|_{L^2} \to 0.$$
Now (\ref{opravdaj1}) immediatelly follows.
(\ref{opravdaj2}) follows from Lemma \ref{lem:1} and triangular inequality. (\ref{opravdaj2}) follows from
the fact that
\begin{align*}
& \int_{S}\int_I \mathcal{Q} (x+tn(x), r(x+tn(x))/\e, \bs E^h_{app}(x+tn(x)) \ud t \ud \mathcal{H}^2 \to \\
& \int_{S}\int_I \int_{\calY} \mathcal{Q} (x+tn(x),y, \bs E(x+tn(x),y)) \ud y \ud t \ud \mathcal{H}^2 \,dx,
\end{align*}
as $h \to 0$ which is the consequence of the continuity of the integral functionals with respect to strong two scale convergence, see \cite{Visintin-06,Visintin-07}.

\end{proof}

We give here a general computation that will be needed in the proof of the next proposition.
Let $P\in C^1(S^1; C^1(\calY; \R^3))$ define $P^h : S^1\to\R^3$ by $P^h = P(\cdot, r/\e)$, where,
as usual, $r$ is extended trivially from $S$ to $S^1$.
Then by \eqref{fladh}
\begin{align*}
\nabla_h P^h &= \frac{1}{h}\dd_n P^h\otimes n + \nabla P^h T_S (I + t\AA)(I + ht\AA)^{-1}T_S
\\
&= \frac{1}{h}\dd_n P(\cdot, r/\e)\otimes n
\\
&+
\left( \nabla P(\cdot, r/\e) + \frac{1}{\e}\nabla_y P(\cdot, r/\e)\nabla_{tan} r\ T_S( I + t\AA )^{-1} \right)(I + t\AA)(I + ht\AA)^{-1}T_S,
\end{align*}
because having extended $r$ trivially to $S^1$, we have $\dd_n r = 0$ and \eqref{dpi-2} applies.
We use the notation $\dd_n P$ to denote the $n$-derivative with respect
to the first argument only. Since $(I + ht\AA)^{-1}$ agrees with $I$ up to a term
that on $S^1$ is uniformly bounded by $h$, and since $h \leq C\e$, we conclude that
\begin{equation}
\label{nablaph}
\left\|\nabla_h P^h - \frac{1}{h}\dd_n P(\cdot, r/\e)\otimes n - \frac{1}{\e}\nabla_y P(\cdot, r/\e)\nabla_{tan} r\, T_S\right\|_{L^{\infty}(S^1)}\leq C.
\end{equation}
Note that the linear operator $\nabla_{tan} r T_S$ on the tangent space can be expressed as, see (\ref{eq:500}),
$$
\nabla_{tan} r T_S = e_1\otimes \tau^1 + e_2\otimes\tau^2,
$$
which is just the pullback operator $\psi^*$ from the tangent space to $\R^2$.

\begin{proposition} \label{prop:peter111}
For every $\B_w \in \mathcal{B}$ and for every infinitesimal bending $V \in H^2(S;\R^{3 \times 3})$ of $S$,
there exists a sequence 
$(\ya^h)\subset H^1(S^1,\R^3)$ 
satisfying the following:
\begin{enumerate}[(i)]
\item $\ya^h\to\pi$ strongly in $H^1(S^{1};\R^3)$.
\item The maps $V^h(x) = h^{-1}\int_I y^h(x + tn(x))\ dt\ - x$ satisfy
$$
\bs V^h\to \V\mbox{ strongly in }H^1(S;\R^3)
$$
and
$$
\tfrac{1}{h} q_{ \bs V^h} \weak \B_w\mbox{ weakly in }L^2(S;\QQ).
$$
\item We have
$$
\lim_{h \to 0} h^{-4} I^h(\ya^h) =
\begin{cases}
I_\gamma (\V,\B_w) &\mbox{ if }\lim h/\e = \gamma\in (0,\infty];
\\
I_0^0 (\V,\B_w) &\mbox{ if }\e\gg h \gg \e^2;
\\
I_{0,\gamma_1}^{1} (\V,\B_w) &\mbox{ if $h\sim \e^2$ with }\lim \e^2/h = 1/\gamma_1.
\end{cases}
$$
\end{enumerate}
\end{proposition}
\begin{proof}
The proof is a modification of the proof of the recovery sequence in \cite{Lewicka1}, cf. also \cite{FJM-06}.

By definition and by density, since ${\B}_w\in \calB$, there exist
$w_n\in C^{\infty}_0(S; \R^3)$ such that $q_{w_n} \to B_w$ strongly in $L^2(S)$.
Hence
$$
\limsup_{n\to\infty}\limsup_{h\to 0} \left( \|q_{w_n} - B_w\|_{L^2(S; \QQ)} + h\|w_n\|_{W^{2,\infty}(S; \R^3)}\right) = 0.
$$
Lemma \ref{L:attouche} yields a sequence $n_h$ with $n_h\to\infty$ as $h\to 0$, such that the maps
$$
w^h = w_{n_h}
$$
satisfy
 \begin{equation}\label{eq:jelgotovo}
 \limsup_{h \to 0} h \| \wa^{h} \|_{W^{2,\infty}(S; \R^3)} = 0,
 \end{equation}
and
\begin{equation}
\label{qw-1}
q_{ \wa^{h}} \to \B_w \mbox{ strongly in } L^2(S; \QQ).
\end{equation}
In order to have common proof for all cases (see the case $\gamma=0$) we will assume that  there exists a constant $M_1 > 0$ such that and double index sequence $q_{w^{\delta,h}}$ such that
\begin{equation}\label{citat1}
\|q_{w^{\delta,h}}\|_{L^2} \leq M_1,
\end{equation}
and for all $\delta>0$ we have
\begin{eqnarray} \label{eq:15}
&&\limsup_{h \to 0} h \|\nabla \wa^{\delta,h} \|_{L^4}=0, \ \limsup_{h \to 0} h \|\nabla^2 \wa^{\delta,h} \|_{L^2} \leq M_1; \\
\nonumber && \limsup_{h \to 0} h^2 \|\nabla \wa^{\delta,h} \|_{L^\infty}=0, \ \limsup_{h \to 0} h^3 \|\nabla^2 \wa^{\delta,h} \|_{L^\infty}=0.
\end{eqnarray}
and
\begin{equation} \label{citat2}
q_{w^{\delta,h}} \rightharpoonup B_w, \text{ weakly in } L^2
\end{equation}

Now let $\V \in H^{2}(S,\R^3)$ be an infinitesimal bending of $S$. We  approximate $\V$ by a sequence $\va^{\delta,h} \in
W^{2,\infty}(S;\R^3)$ such that, for each
$\delta>0$ we have:
\begin{eqnarray}
\nonumber & &\limsup_{h \to 0} \|\va^{\delta,h}-\V\|_{H^2(S)}=0,\quad \limsup_{h \to 0} h\|\va^{\delta,h}\|_{W^{2,\infty}(S)} \leq \delta \\
& & \label{eq:16}\\
& &\nonumber \lim_{h \to 0} \frac{1}{h^2} \HH^2\left(\{ x \in S : \va^{\delta,h}(x) \neq \V(x) \}\right) = 0.
\end{eqnarray}
The existence of such $\va^{\delta,h}$ follows by Proposition  \ref{P:FJM} in
the appendix.
We claim that
\begin{equation}\label{eq:18}
\frac{q_{\va^{\delta,h}}}{h} \to 0 \mbox{ strongly in }L^2(S)
\end{equation}
as $h\to 0$. In fact,

\begin{eqnarray}
\label{eq:20}\frac{1}{h} \|q_{ \va^{\delta,h}}\|_{L^2(S)} &\leq& \frac{1}{h} \left|\{ x \in S; \va^{\delta,h}(x) \neq \V(x)\}\right|^{1/2}
\cdot \|\nabla_{tan}{ \va^{\delta,h}}\|_{L^{\infty}(\{v^{\delta, h}\neq V\})},
\end{eqnarray}
and this converges to zero by \eqref{eq:16} and because $q_{v^{\delta, h}}$ is uniformly bounded in $L^{\infty}(S)$ (see below).
From this we deduce (\ref{eq:18}).
\\
To see that $q_{v^{\delta, h}}$ is uniformly bounded in $L^{\infty}(S)$,
note that the Lipschitz constants of all $q_{\va^{\delta,h}}$ are bounded by ${\delta}/{h}$.
Since $q_{v^{\delta, h}} = 0$ almost everywhere on $\{\va^{\delta,h} = V\}$, we have
$$
|q_{ \va^{\delta,h}}(x)| \leq C \tfrac{1}{h}\dist\big(x,\{\va^{\delta,h}=\V\} \big) \leq C.
$$
The last estimate is true because (due to \eqref{eq:16} and bounded curvature of $S$) for small $h$
the set $\{v^{\delta, h} \neq V\}$ cannot contain a disk of radius $h$.

Now let $o^{\delta}, p^\delta \in C^1(S^1;C^1(\calY; \R^3))$ and set $p^{\delta, h} = p^{\delta}(\cdot, r/\e)$
and $o^{\delta,h} = o^{\delta}(\cdot, r/\e)$.

We define $z^{\delta, h} : S^h\to\R^3$ by
$$
z^{\delta, h} = id + h\left( v^{\delta, h} + hw^{\delta, h}+ t\mu_{v^{\delta, h} + hw^{\delta, h}}\right) + h^3p^{\delta, h}( \Phi^h )
+ \e h^2 o^{\delta, h}(\Phi^h),
$$
and we define $y^{\delta, h} : S^1\to\R^3$ by $y^{\delta, h}(\Phi^h) = z^{\delta, h}$.
Clearly, for each $\delta$, as $h\to 0$ we have
$$
y^{\delta, h} \to \pi \mbox{ strongly in }H^1(S^1).
$$
For $x\in S$ we define
$$
V^{\delta, h}(x) = \frac{1}{h^2}\left(\int_{I^h} z^{\delta, h}(x + tn(x))\ dt - x\right) = \frac{1}{h}\left(\int_I y^{\delta, h}(x + tn(x))\ dt\ - x\right).
$$
Hence
$$
\V^{\delta,h} = \va^{\delta,h} + h\wa^{\delta,h} + \e h \t{o}^{\delta} + h^2 \t p^\delta.
$$
Here we have introduced $\t o^{\delta, h} : S\to\R^3$ by
$$
\t{o}^{\delta, h}(x) = \int_{I} o^{\delta, h}\left( x + tn(x) \right)\ dt,
$$
and $\t{p}^{\delta, h}$ is defined similarly.

From  \eqref{citat2} and (\ref{eq:18})  we deduce that for each $\delta$, as $h\to 0$ we have
\begin{equation} \label{eq:iggg2119}
\frac{1}{h} q_{V^{\delta,h}} = \frac{1}{h}q_{\va^{\delta,h}} + q_{\wa^{\delta,h}} + \e q_{\t{o}^{\delta}} + h q_{\t p^\delta}
\rightharpoonup \B_w\mbox{ weakly in } L^2(S;\QQ).
\end{equation}
Thus the first two parts of the claim are satisfied (for each $\delta$).
Also from  (\ref{eq:18}) we deduce
\begin{eqnarray} \label{eq:iggg2122}
&& \limsup_{h \to 0} \left\| \frac{1}{h}  q_{\V^{\delta,h}}\right\|_{L^2} \leq R\Big( \limsup_{h\to 0}\|q_{ \wa^{\delta,h}}\|_{L^2}+\|\bs o^\delta\|_{L^2(S;H^1(I \times \calY;\R^3))} \\ \nonumber && \hspace{+25ex}+ \limsup_{h\to 0}h\|q_{\int_I \bs p^\delta \ud t}\|_{L^2}\Big),
\end{eqnarray}
for some $R>0$. Notice that for $\wa^{\delta,h}$ that satisfies (\ref{qw-1})  we have for all $\delta>0$
\begin{equation}\label{anna100}
 \limsup_{h\to 0}\|q_{\wa^{\delta,h}}\|_{L^2} =\|\B_w\|_{L^2}.
\end{equation}

In order to prove the third part, we need to understand the limiting behaviour of
$$
E^{\delta,h} = \frac{1}{h^2}\left(\sqrt{ (\nabla_h y^{\delta, h})^T\nabla_h y^{\delta, h}}-\id\right).
$$
Define $Q : S^h\to\R^{3\times 3}$ by \eqref{defQ}. Lemma \ref{ansatz} shows that
\begin{align*}
\nabla_h y^{\delta, h}(\Phi^h) = \nabla z^{\delta, h} = I + h \O_{v^{\delta,h}} + h^2G^{\delta, h}(\Phi^h),
\end{align*}
where $G^{\delta, h} : S^1\to\R^{3\times 3}$ is defined by the following equation on $S^h$:
\begin{align*}
G^{\delta, h}(\Phi^h) &= \Omega_{w^{\delta, h}} - \frac{t}{h} b_{v^{\delta, h}}(T_S, T_S) - tb_{w^{\delta, h}}(T_S, T_S)
- \frac{t^2}{h}\nabla_{tan}\mu_{v^{\delta,h} + hw^{\delta, h}}\AA
\\
&+ \frac{1}{h}\left(\nabla_{tan} v^{\delta,h} + h\nabla_{tan}w^{\delta, h} + t\nabla_{tan}\mu_{v^{\delta, h} + hw^{\delta, h}} \right)T_S Q
\\
&+ h\nabla_h p^{\delta, h}(\Phi^h) + \e \nabla_h o^{\delta, h}(\Phi^h).
\end{align*}

 In the case $\gamma = \lim h/\e$ is nonzero, we deduce from \eqref{nablaph} (applied with $P = p^{\delta}$ and $P = o^{\delta}$) that, as $h\to 0$,
\begin{equation}
\label{phoh}
 h\nabla_h p^{\delta, h} + \e \nabla_h o^{\delta, h} \stto \Xi^{\delta}.
\end{equation}
Here we have introduced $\Xi^{\delta} : S^1\times\calY \to \R^{3\times 3}$ by
$$
\Xi^{\delta}(x, y) = \left(\d_n p^{\delta}(x, y) + \frac{1}{\gamma} \d_n o^{\delta}(x, y)\right)\otimes n(x)
+ \left( \gamma\nabla_y p^{\delta}(x, y) + \nabla_y o^{\delta}(x, y)\right) \nabla_{tan} r(x)\, T_S(x).
$$
Observe that \eqref{phoh} remains true for $\gamma = 0$ provided that $\d_n o^{\delta} \equiv 0$,
and for $\gamma = \infty$ provided that $p^{\delta, h}$ does not depend on $y$. This follows from \eqref{nablaph}.
\\
Lemma \ref{irre} and \eqref{qw-1} shows that
\begin{equation}
\label{irre-4}
\sym\Omega_{w^{\delta, h}} = q_{w^{\delta, h}}(T_S, T_S) \to B_w(T_S, T_S) \mbox{ strongly in }L^2(S).
\end{equation}

And
$$
b_{v^{\delta,h}}(T_S, T_S) \to b_V(T_S, T_S)\mbox{ in }L^2(S).
$$
This holds for each $\delta>0$ as $ h \to 0$.

We conclude that
\begin{equation}
\label{irre-5}
\sym G^{\delta, h} \stto B_w(T_S, T_S) - tb_V(T_S, T_S)
+ \sym\Xi^{\delta}
\end{equation}
strongly two-scale in $S^1$ as $h\to 0$. Hence the map $E^{\delta, h}_{app} : S^1\to\R^{3\times 3}$ defined by
$$
E^{\delta,h}_{app} = \sym\bs G^{\delta,h} - \frac{1}{2}(\O_{v^{\delta, h}})^2
$$
converges strongly two-scale on $S^1$ to $E^{\delta} : S^1\times\calY\to\R^{3\times 3}$, given by
$$
E^{\delta}(x, y) = B_w(x)(T_S(x), T_S(x)) - t(x)b_V(x)(T_S(x), T_S(x)) - \tfrac{1}{2} \bs \Omega_V^2(x) + \sym\Xi^{\delta}(x, y).
$$
We now wish to apply Lemma \ref{lem:3} to $E^{\delta, h}$ and $E^{\delta, h}_{app}$ in order to conclude that

\begin{eqnarray}
\label{irre-6}
&& h^{-4} \int_{S}\int_I W\left( x+tn(x), r(x+tn(x))/\e, I + h^2E^{\delta, h}(x+tn(x)) \right)\ud t \ud \mathcal{H}^2 \to \\ \nonumber&&\int_{S}\int_I\int_{\calY} \mathcal{Q} (x+tn(x), E^{\delta}(x+tn(x),y))\ dy\ dt \ d\mathcal{H}^2 + \tilde{E}(\delta),
\end{eqnarray}
where $|\tilde{E}(\delta)| \to 0$ as $\delta\to 0$. In order to prove this, it remains to
verify that the hypotheses of Lemma \ref{lem:3} are indeed satisfied.
\\
But in fact, it is not difficult to see that there exists a constant $M$ such that for all $\delta>0$ small enough
\begin{equation} \label{philippe}
\limsup_{h \to 0} \left(\|\sym G^{\delta, h}\|_{L^2(S^1)} + \|\O_{v^{\delta, h}}\|_{L^4(S)}\right) \leq M
\end{equation}
This follows from (\ref{eq:iggg2122}) and the construction of $p^\delta$, $o^\delta$ below.
Moreover, hypotheses (ii) and (iii) of Lemma \ref{lem:3} are clearly satisfied by virtue of Sobolev embedding
and by (\ref{eq:15}), (\ref{eq:16}), (\ref{eq:18}). Hypothesis (iv) is a direct consequence of
(\ref{eq:15}) and that, as $h\to 0$,
\begin{align*}
h\| \nabla^2 \va^{\delta,h} \|_{L^4} &\leq Ch \sqrt{\|\nabla^2 \va^{\delta,h} \|_{L^\infty}} \sqrt {\|\nabla^2 \va^{\delta,h} \|_{L^2}}\to 0,
\\
h^2\| \nabla^2 \bs w^{\delta,h} \|_{L^4} &\leq h^2\sqrt{\|\nabla^2  \bs w^{\delta,h} \|_{L^{\infty}}} \sqrt {\|\nabla^2 \bs w^{\delta,h} \|_{L^2}} \to 0,
\end{align*}
which follow from (\ref{eq:15}) and (\ref{eq:16}). Hence the hypotheses of Lemma \ref{lem:3} are indeed satisfied, and
\eqref{irre-6} follows.

It remains to choose the oscillations $p^{\delta}$ and $o^{\delta}$ in an optimal way. We have to distinguish the three cases.

 \textit{The case $\gamma \in  (0,\infty)$.}
Let $A_{\delta} \in C^2 (S; \so 3)$ be such that $A_{\delta} \to \O_V$ strongly in $H^1(S)$ as $\delta\to 0$. Define
$p^{\delta} : S^1\times\calY\to\R^3$ by
$$
p^{\delta}(x, y) = t(x)\left( \frac{|A_{\delta}(x) n(x)|^2}{2}\ I + A_{\delta}^2(x) \right)n(x).
$$
Then clearly $\nabla_yp^{\delta} \equiv 0$. We claim that
\begin{equation}
\label{irre-8}
\sym\left( \d_n p^{\delta}\otimes n \right) - \frac{1}{2}A_{\delta}^2 = -\frac{1}{2}A_{\delta}^2(T_S, T_S).
\end{equation}
In fact,
\begin{align*}
A_{\delta}^2 - T_S A_{\delta}^2 T_S &= 
(n\otimes n)A_{\delta}^2 - (n\otimes n)A_{\delta}^2 (n\otimes n) + A_{\delta}^2 (n\otimes n)
\\
&= 2\sym\left( A_{\delta}^2 n \otimes n + \frac{|A_{\delta} n|^2}{2}\ n\otimes n\right)
\end{align*}
because $A_{\delta}$ is skew symmetric. And this equals $2\sym(\d_np^{\delta}\otimes n)$ because clearly
$$
\d_np^{\delta} = \frac{|A_{\delta} n|^2}{2}n + A_{\delta}^2n.
$$
Thus \eqref{irre-8} is proven.
\\
From \eqref{irre-8} and from the definition of $E^{\delta}$ and $\Xi^{\delta}$ we conclude:
$$
E^{\delta} = (B_w - tb_V - \frac{1}{2}(A^{\delta})^2)(T_S, T_S) + \tilde{\mathcal{U}}_{\gamma}(o^{\delta}),
$$
where
$$
\tilde{\mathcal{U}}_\gamma(o^\delta) := \sym\left( \nabla_yo^{\delta}\nabla_{tan}rT_S + \gamma^{-1}\d_n o^{\delta}\otimes n \right).
$$
Notice that $\tilde{\mathcal{U}}_\gamma(o^\delta)=\mathcal{U}_\gamma(\tilde{o}^\delta)$,
which is just $\mathcal{U}_\gamma$ as in Definition~\ref{def:relfield} and
$\tilde{o}^\delta_i=o^\delta \cdot \taub_i$.
Now we choose ${\tilde o}^\delta \in C^1(S^1;\dot{C}^1(\calY;\R^3))$ in such a way that
\begin{equation}
\mathcal{U}_\gamma(\tilde{o}^\delta) \to \Pi_\gamma\big(\B_w + \tfrac{(dV)^2}{2}, -b_V\big)\mbox{ strongly in }L^2(S;\R^{3 \times 3})
\end{equation}
as $\delta\to 0$.
Here, the operator $\Pi_\gamma$ is as in Lemma~\ref{lem:2a}. Then
\begin{equation}
\label{limedelta}
E^{\delta} \to (B_w - tb_V)(T_S, T_S) + \frac{(dV)^2}{2} + \Pi_\gamma\big(\B_w + \frac{(dV)^2}{2}, -b_V\big)\mbox{ strongly in }L^2(S;\R^{3 \times 3}),
\end{equation}
as $\delta \to 0$ because $- A_{\delta}^2(T_S, T_S) \to (dV)^2$ by \eqref{form-2}.
\\
By \eqref{limedelta} and by the above results, we see that
\begin{eqnarray*}
g(\delta,h) = \|\ya^{\delta,h}-\pi \|_{H^1(S^1)}+\|\bs V^{\delta,h}-\bs V\|_{H^1(S)}+\tilde{d}^K(\tfrac{1}{h} q_{ \V^{\delta,h}}, \B_w) +
\left| \tfrac{1}{h^4} I^h(\ya^h) - I_\gamma (\V,\B_w) \right|
\end{eqnarray*}
satisfies $
\limsup_{\delta \to 0} \limsup_{h\to 0} g(\delta,h)=0.
$
Here $\tilde{d}^K: L^2(S;\Ma^3) \times L^2(S;\Ma^3) \to \R $ is defined in the following way: For $K>0$
we know there exists a metric $d^K$ which defines the weak topology on the ball of radius $K$. We define:
$$
\tilde{d}^K(\Ma_1,\Ma_2)=\left\{\begin{aligned}d^K(\Ma_1,\Ma_2), & \text{ if } \|\Ma_1\|_{L^2} <K \text{ and } \|\Ma_2\|_{L^2}<K,\\ +\infty,& \text{ otherwise.}    \end{aligned} \right.
$$
The constant $K=K(R,\|\B_w\|_{L^2}, \|\bs V\|_{H^2})$ we choose in a way that the right hand side of
(\ref{eq:iggg2122}) is bounded by e.g. $\tfrac{K}{2}-1$.

Lemma \ref{L:attouche} then yields a sequence $\delta_h\to 0$ such that $g(\delta_h, h) \to 0$ as $h\to 0$.

 \textit{The case $\gamma=\infty $.}
Let $A_{\delta} \in C^2 (S; \so 3)$ be such that $A_{\delta} \to A$ strongly in $H^1(S)$ as $\delta\to 0$. Define
$\zetaa^\delta \in C^1(S;\dot{C}^1(I\times \calY;\R^2))$, $\psi^\delta \in C^1(S;\dot{C}^1(I \times \calY))$, $\bs c^\delta \in C^1(S;\dot{C}^1(I;\R^3))$ such that
\begin{equation*}
\mathcal{U}_\infty(\zetaa^\delta, \psi^\delta, \bs c^\delta)\to \Pi_\infty\big(\B_w + \tfrac{(dV)^2}{2},-b_V\big)\mbox{ strongly in }L^2(S;\R^{3 \times 3})
\end{equation*}
as $\delta\to 0$.
We will use the following fact: if $f : I\times \calY \to\R^3$ then $F(x, y) = \int_0^{t(x)} f(s, y)\ ds$ satisfies $\d_n F(x,y) = f(t(x), y)$.
Again we wish to have $p^{\delta}$ independent of $y$, in order to ensure the validity of \eqref{phoh}. We define
$$
p^{\delta}(x, y) = t(x)\left( \frac{|A_{\delta}(x) n(x)|^2}{2}\ I + A_{\delta}^2(x) \right)n(x) + 2\int_0^{t(x)} c_\alpha(s)\ ds \taub^\alpha+ \int_0^{t(x)} c_3(s)\ ds \taub^3.
$$
Then
$$
\d_n p^{\delta} = \frac{|A_{\delta}n|^2}{2}n + A_{\delta}^2 n + 2c_\alpha(t)\taub^\alpha+c_3(t)\taub^3.
$$
For $x\in S^1$ and $y\in\calY$ set
$$
o^{\delta}(x, y) = \zeta^{\delta}_{\alpha}(\pi(x), t(x), y)\tau^{\alpha}(x) + 2\psi^{\delta}(\pi(x), t(x), y) n(x).
$$
Then
$$
\nabla_y o^{\delta} =   \tau^{\alpha} \otimes \nabla_{y}\zeta^{\delta}_{\alpha} + 2  n \otimes \nabla_y\psi^{\delta}.
$$
Since $\gamma = \infty$, we have
$$
\sym\Xi^{\delta} =
\sym\left( \d_n p^{\delta} \otimes n + \nabla_y o^{\delta} \nabla_{tan} r\, T_S\right) = \mathcal{U}_\infty(\zetaa^\delta, \psi^\delta, \bs c^\delta),
$$
From now on the proof is analogous to the case $\gamma \in (0,\infty)$.

\textit{Construction for $\gamma=0$ and $\lim_{h \to 0}\tfrac{\eh^2}{h} \in [0,\infty) $ }:

In this one has to modify, in addition, the maps $\wa^{\delta,h}$.
Let $A_{\delta}$ be as before. In the case $\lim_{h \to 0} \tfrac{\eh^2}{h}=0$ choose
$\zetaa^\delta \in C^1(S;\dot{C}^1(\calY;\R^2))$, $\varphi^\delta \in C^1(S;\dot{C}^2(\calY))$, $\bs g^\delta \in C^1(S;\dot{C}^1(I \times \calY;\R^3))$ such that
\begin{equation}
\mathcal{U}^0_0(\zetaa^\delta, \varphi^\delta, \bs g^\delta) \to \Pi_0^0 \big(\B_w + \frac{(dV)^2}{2},-b_V\big)\mbox{ strongly in }L^2(S;\R^{3 \times 3}),
\end{equation}
as $\delta\to 0$. In the case $\lim_{h \to 0} \tfrac{\eh^2}{h}=\tfrac{1}{\gamma_1} \in (0,\infty)$ choose them such that
\begin{equation}\label{anna2000}
\mathcal{U}^{1}_{0,\gamma_1}(\zetaa^\delta, \varphi^\delta, \bs g^\delta)\to
\Pi_0^{1,\gamma_1} \big(\B_w + \frac{(dV)^2}{2},-b_V\big)\mbox{ strongly in }L^2(S^1).
\end{equation}
Extend $\zeta^{\delta}$ trivially to $S^1$ and define $o^{\delta}(x, y) = \zeta^{\delta}_{\alpha}(x, y)\tau^{\alpha}(x)$. Then
$\d_n o^{\delta} \equiv 0$, so \eqref{phoh} remains true. We define
\begin{eqnarray*}
p^{\delta}(x, y) &=& t(x)\left( \frac{|A_{\delta}(x) n(x)|^2}{2}\ I + A_{\delta}^2(x) \right)n(x) +2 \int_0^{t(x)} g_\alpha^{\delta}(\pi(x), s, y)\ ds\ \tau^\alpha(x) \\ & &+\int_0^{t(x)} g_3^{\delta}(\pi(x), s, y)\ ds\ \tau^3(x).
\end{eqnarray*}
We define the modified fields
\begin{eqnarray*}
 \tilde{\wa}^{\delta,h}&=&\wa^{\delta,h} +\tfrac{\e^2}{h} \varphi^\delta(\cdot,r/\e) \na,
\end{eqnarray*}
where $\wa^{\delta,h}$ is defined by the property (\ref{eq:jelgotovo}).

Notice that $\tilde{\wa}^{\delta,h}$ satisfies the condition (\ref{eq:15}) with $M_1\leq C \| \nabla^2_y \varphi^\delta \|$, for some constant $C>0$, independent of $\delta$.
Also using the following facts valid for every fixed $\delta>0$:
\begin{align}
\frac{h}{\e}\left( \nabla\t w^{\delta, h} - \nabla w^{\delta, h} \right) &  \mbox{ bounded in }L^{\infty}(S)
\\
\frac{h}{\e}\left( \Omega_{w^{\delta,h}}-\Omega_{\tilde{w}^{\delta,h}} \right) & \mbox{ bounded in }L^{\infty}(S)
\\
\frac{h}{\e}\left(b_{\tilde{w}^{\delta,h}}-\left(b_{w}^{\delta,h}+ \tfrac{1}{h}\sum_{i,j=1}^2 \partial_{y_i y_j} \varphi^\delta \taub^i \otimes \taub^j\right) \right) &  \mbox{ bounded in }L^{\infty}(S) \\
\label{citat3} q_{\tilde{\wa}^{\delta,h}} = q_{\wa^{\delta,h}}+\tfrac{\e^2}{h} \varphi^\delta(\cdot,r/\e)  \AA & \stto B_w+\tfrac{1}{\gamma_1} \varphi^\delta(x,y)\AA,
\end{align}
 we can repeat the same argument as in the case $\gamma \in (0,\infty)$.
Namely, notice that (\ref{citat1}) is valid and thus  the right hand side of (\ref{eq:iggg2122}) can be bounded, by a bound independently of $\delta$. It can be easily seen that (\ref{eq:15}), (\ref{citat2}) are valid. Instead of (\ref{irre-4}) we have (\ref{citat3}).

\end{proof}

\section{Convex shell} \label{sesti}
In this chapter we shall identify the $\Gamma$-limit for convex shells in the remaining case,
i.e. $h\ll\e^2$.
We want to demonstrate the stronger influence of the geometry in this case. We work under the assumption that there exists $C>0$ such that
\begin{equation}\label{asss2}
 \AA(\hx)
\taub \cdot \taub  \geq C  \taub \cdot \taub , \forall \hx \in S, \taub \in T_{\hx} S.
\end{equation}

\begin{definition}\label{def:relfieldcon}
  For $\hx \in S$ we define the following operator
  \begin{eqnarray*}
 && \mathcal{U}_0^{2,c}: \dot{L}^2(\calY;\R^{2\times 2}_{\sym})   \times  L^2(I\times\mathcal Y;\R^3) \to L^2 (I \times\mathcal Y;\R^{3 \times 3}_{\sym}), \\
  && \mathcal{U}^{2,c}_0 (\B,\bs g)= \sum_{i,j=1}^3 \left(\begin{array}{cc}
        \B
        & \begin{array}{cc} \bs g_1\\ \bs g_2\end{array}\\
        (\bs g_1,\;\bs g_2)&\bs g_3\\
      \end{array}\right)_{ij} \taub^i \otimes \taub^j\,
\end{eqnarray*}
and  function space of \textit{relaxation fields}
  \begin{equation}
  L_0^{2,c}(I\times\mathcal Y;\R^{3 \times 3}_{\sym}):=\left\{ \mathcal{U}_0^{2,c}(B,g)\,: B\in \dot{L}^2(\calY;\R^{2\times 2}_{\sym}),\ g \in L^2(I\times\mathcal Y;\R^3)\right\}.
  \end{equation}
\end{definition}
Again as before it can be seen that $L_0^{2,c}(I\times\mathcal Y;\R^{3 \times 3}_{\sym})$ is a
closed subspace of $L^2(I \times \calY;\R^{3 \times 3}_{\sym})$. We also define the functional $I_0^{2,c}:H^2(S;\R^3) \times L^2(S;\mathbb{S}))\to \R$
\begin{equation} \label{eq:functional}
I_0^{2,c} (\V,\B_w)=\int_S \mathcal{Q}_0^{2,c}(\hx,\B_w+\tfrac{1}{2} (dV)^2,\, -b_V \big) \ud \hx,
\end{equation}
with the quadratic form $ \mathcal{Q}_0^{2,c}(x):\mathbb{S}(x)_{\sym} \times \mathbb{S}(x)_{\sym} \to \R$:
\begin{equation} \label{defQgi111111}
 \mathcal{Q}_0^{2,c}(x,q^1,q^2)= \inf_{\bs U \in L_0^{2,c} (I \times \calY; \R^{3 \times 3}_{\sym} ) }
\iint_{I \times \mathcal{Y}} \mathcal{Q}\Big(x+tn(x),y,q^1+tq^2+\bs U \Big) \ud t \ud y.
\end{equation}
As before, it is easy to see that the definition is equivalent to the following one:
\begin{eqnarray} \label{defQgi111111111}
 \mathcal{Q}_0^{2,c}(x,q^1,q^2)&=& \\\nonumber & & \hspace{-20ex} \inf_{\dot{\B} \in \dot{L}^2(\calY;\M^2_{\sym})}
\iint_{I \times \mathcal{Y}} \mathcal{Q}_2\Big(x+tn(x),y,q^1+tq^2+\sum_{i,j=1}^2 \dot{B}_{ij} \taub_i \otimes \taub_j) \ud t \ud y.
\end{eqnarray}

In the case when $\mathcal{Q}$ does not depend on $t$ we have that
\begin{eqnarray*}
\tilde{\mathcal{Q}}_0^{2,c}(x,q^1,q^2) &=& \inf_{ \dot{\B} \in \dot{L}^2(\calY;\M^2_{\sym}) }\int_{\mathcal{Y}} \mathcal{Q}_2(x,y,q^1+\sum_{i,j=1}^2(\dot{ \B})_{ij}\taub^i \otimes \taub^j ) \ud y\\ \nonumber & & + \frac{1}{12}
\int_{\mathcal{Y}} \mathcal{Q}_2 (x,y,q^2) \ud y.
\end{eqnarray*}
Under the assumption (\ref{asss2}) it is well-known that
$\mathcal{B} = L^2(S,\mathbb{S})$, cf. e.g. \cite{Ciarletshell00}, \cite{Lewicka1}.
Thus if one wants to additionally to relax the functional $I_0^{2,c}$ with respect to $\B_w$, one in this case obtains the functional
 $ \tilde{I}_0^{2,c}:H^2(S;\R^3)\to \R$
\begin{equation} \label{eq:functional1}
 \tilde{I}_0^{2,c} (\V)=\frac{1}{12} \int_S \int_{\mathcal{Y}} \mathcal{Q}_2(\hx,y, b_V)\ud y \ud \hx.
\end{equation}
This functional is the same as in the ordinary von K\'arm\'an model.
For the form $\mathcal{Q}_0^{2,c}$ one can make assertions analogous
to Lemma \ref{lem:2a} with the appropriate operator $\Pi_0^{2,c}$ and Lemma \ref{lem:2b}.
We introduce the space
\begin{align*}
FL(S; \dot{C}^{\infty}(\calY)) =
\Big\{(x, y)\mapsto
\sum_{k\in \Z^2, \ |k| \leq n, \ k \neq 0} & c^k(x) e^{2 \pi i  k \cdot  y} :
\\
& n \in \mathbf{N} \mbox{ and }c^k \in C_0^1(S;\mathbb{C})\mbox{ with }
\overline{c}^k=c^{-k}
\Big\}.
\end{align*}
By Fourier transform it can be easily seen that $FL(S; C^{\infty}(\calY))$ is dense in $L^2(S;\dot{H}^m(\calY))$, for any $m\in \N_0$.
The following lemma resembles Lemma 3.3 in \cite{Schmidt-07}:
\begin{lemma} \label{lm:anna100}
Assume (\ref{asss2}) and let $\dot{\B} \in L^2(S;\dot{L}^2(\calY;\mathbb{S}))$.
Then there exist unique $\wa \in L^2(S;\dot{H}^1(\calY;\R^2))$ and $\varphi \in L^2(S;\dot{L}^2(Y))$ such that
\begin{equation} \label{zasrrr}
\sum_{i,j=1,2}\big(\sym \nabla_y \wa\big)_{ij}\taub^i \otimes \taub^j+\varphi \AA=\dot{\B}.
\end{equation}
Moreover, if $\dot{\B}_{ij} \in FL(S; \dot{C}^{\infty}(\calY))$ for every $i,j=1,2$ then
$\wa_i \in FL(S;\dot{H}^1(\calY))$, for $i=1,2$ and $\varphi \in FL(S;\dot{H}^1(\calY))$
\end{lemma}
\begin{proof}
One possible proof is to apply the operator $\curl_y\curl_y$ to both sides of \eqref{zasrrr},
which leads to the PDE
$$
\cof\AA : \nabla^2_y\p = \curl_y\curl_y \dot{B},
$$
which by virtue of \eqref{asss2} is an elliptic PDE with constant coefficients (for each $x$). We prefer to give a direct proof.
\\
There exist $b_{ij}^k$ such that for all $i,j=1,2$:
\begin{eqnarray*}
&&\dot{\B}(\hx,y)_{ij}=\sum_{k\in \Z^2} b^k_{ij}(\hx)  e^{2 \pi i  k \cdot y}, \text{ where } \sum_{k\in \Z^2}\| b^k_{ij} \|_{L^2} <\infty, \overline{b}^k_{ij}=b^{-k}_{ij},\ b^0_{ij}=0.
\end{eqnarray*}
We assume that for $i=1,2$:
$$
\wa_j= \sum_{k \in \Z^2} c^k_{j}(\hx)  e^{2 \pi i  k \cdot y},\ \overline{c}^k_{j}=c^{-k}_{j},\ c^0_j=0
$$
and
$$ \varphi=\sum_{k \in \Z^2} d^k(\hx)  e^{2 \pi i k \cdot  y },\ \overline{d}^k=d^{-k}, \ d^0=0. $$
The equation (\ref{zasrrr}) is equivalent to the following
problem for every $(k_1,k_2) \in \Z^2$ find complex coefficients $c_j^k$, $b_{ij}^k$, $d^k$ such that
\begin{eqnarray*}
k_1 c^k_1+d^k \AA_{11}&=& b^k_{11} \\
\frac{1}{2}( k_2 c^k_1+k_1 c^k_2)+d^k\AA_{12}&=& b^k_{12} \\
k_2 c^k_2+d^k \AA_{22}&=& b^k_{22}.
\end{eqnarray*}

By (\ref{eq:0}) and (\ref{asss2}) it is easy to see that there exists $C>0$ such that the determinant of the system is bounded from below
by
$C (k_1^2+k_2^2). $
Using this it follows that there exists $C>0$ such that
$$ | d^k(\hx)|^2+\sum_{i=1}^2 |k|^2 |c^k_i(\hx)|^2 \leq C(\sum_{i,j=1}^2 |b^k_{ij}(\hx)|^2), \forall \hx \in S. $$
Now all claims follow easily.
\end{proof}

\begin{theorem}\label{tm:glavni2}
Assume (\ref{asss2}) and that $h\ll\e^2$ and that $W$
satisfies Assumption \ref{ass:main}. Let $(\ua^h)\subset H^1(S^h;\R^3)$ satisfy
\begin{equation} \label{uuuuuvjet}
\limsup_{h \to 0} \tfrac{1}{h^4} E^h(\ua^h) < \infty,
\end{equation}
where
\begin{eqnarray*}
E^h(\bs u^h)&=&\frac{1}{h}\int_{S^h} \bar{W}(\Phi^h(x), \ra(x)/\e , \nabla \ua^h)dx.
\end{eqnarray*}
Define $\ya^h\in H^1(S^{1};\R^3)$ by the equation
$
\bar{\ya}^h (\Phi^h(x))=\ua^h(x)
$.
Then the following are true
\begin{enumerate}[(i)]
\item \textit{(compactness)}. There exists a subsequence of $(\bar{\ya}^h)$, still denoted by $(\bar{\ya}^h)$ and there exist $\bs Q^h \in \SO 3$ and  $ c^h \in \R^3$ such that the sequences
$\ya^h:=(\bs Q^h)^T\bar{\ya}^h-\bs c^h$ and
$\V^h :=\frac{1}{h} \left(\int_I \ya^h(\hx+t\na(\hx)) \ud t-x\right)$
satisfy the following
\begin{enumerate}[(a)]
\item $\ya^h \to \pi$ strongly in $H^1(S^{1};\R^3)$.
\item There exists an infinitesimal bending $V \in H^2(S;\R^3)$ of $S$ such that
$
\V^h \to \V \text{ strongly in } H^1(S;\R^3).
$
\item
There exists $B_w \in L^2(S;\mathbb{S})$ such that
 $$\tfrac{1}{h} q_{\V^h} \rightharpoonup \B_w  \text{ weakly in } L^2(S; \mathbb{S}).$$
    \end{enumerate}
\item   \textit{(lower bound)}. Defining $I_0^{2,c}$ by (\ref{eq:functional}) we have
$$
\liminf_{h \to 0} \tfrac{1}{h^4} E^h (\ua^h) \geq I_0^{2,c} (\V,\tilde{\B}_w),
$$
\item \textit{(recovery sequence)} For any infintesimal bending $\V \in H^2(S^{1},\R^3)$
and $\B_w \in L^2(S;\mathbb{S})$ there exists $\ua^h \in H^1(S^h;\R^3)$
satisfying (\ref{uuuuuvjet}) and such that conclusions of part (i) are true with $\bs Q^h=\I$ and $\bs c^h=0$.
Moreover, equality holds in (ii).
\end{enumerate}
\end{theorem}
\begin{proof}
We will only give the sketch of the proof since it is analogous to the previous cases.
Since $V$ is an infinitesimal bending, we have
\begin{equation} \label{anna111}
\partial_{\taub} \V(\hx)=\A(\hx) \taub, \text{ for all } \taub
\in T_{\hx} S \text{ for some } \A \in H^1(S;\so 3),
\end{equation}
Let us assume as in
Proposition \ref{tm:ggggg} that
$$
\nabla_{\ttan} \nabla_{\ttan} (\V^h_s \cdot \na) \wtto \nabla_{\ttan} \nabla_{\ttan}(V \na)+ \sum_{i,j=1}^2 (\partial^2_{y_iy_j} \varphi)\taub^i \otimes \taub^j,
$$
for some $\varphi \in L^2(S;\dot{H}^2(\calY)$. Using (iii) of Proposition \ref{tm:1} as well as Lemma \ref{lm:anna100} we conclude that $\varphi=0$. Thus from
Proposition \ref{tm:ggggg} we conclude that for $\bs E^h$ defined in (\ref{p36-1444}) we have $\bs E^h \wtto \bs E(x,y)$ where
\begin{eqnarray}\label{eq:anna235}
\bs E &=&\B_{w}+\dot{B}+\tfrac{1}{2} (dV)^2 -tb_V+\bs U,
\end{eqnarray}
for some   $\bs U \in  L^2(S; L_0(I\times\mathcal Y))$ and $\dot{B} \in L^2(S;\dot{L}^2(\calY;\mathbb{S}))$. The lower bound easily follows from Lemma \ref{lm:iggg11111} and the definition of the functional $I_0^{2,c}$.

To prove the upper bound we follow the proof of Proposition \ref{prop:peter111}, the case $\gamma=0$. Namely, let us again take   $\bs A^{\delta} \in C^2 (S; \so 3)$  such that $\lim_{\delta \to 0}\|\bs A^{\delta}-\bs A\|_{H^1}=0$ and  $\dot{\B}^\delta$ such that for every $i,j=1,2$, $(\dot{\B}^\delta)_{ij} \in FL(S;\dot{C}^\infty(\calY))$,  $\bs g^\delta \in C^1(S;C^1(I \times \calY;\R^3))$ and
\begin{equation}
\lim_{\delta \to 0}\left\| \mathcal{U}^{2,c}_0(\dot{\B}^\delta, \bs g^\delta)-\Pi_0^{2,c} \big(\B_w+\tfrac{1}{2} (dV)^2,-b_V\big)\right\|_{L^2(S;\R^{3 \times 3})}=0.
\end{equation}
By Lemma \ref{lm:anna100} there exist
$\bs z^\delta \in (FL(S;\dot{C}^\infty(\calY)))^2$ and $\varphi^\delta \in FL(S;\dot{C}^\infty(\calY) $
 solving the system
 \begin{equation}
\sum_{i,j=1,2}\big(\sym \nabla_y \bs z^\delta \big)_{ij}\taub^i \otimes \taub^j+\varphi^\delta \AA_{ij}=\dot{\B}^{\delta}.
\end{equation}

We define
\begin{eqnarray*}
&& \tilde{\va}^{\delta,h}=\va^{\delta,h}, \\
&& \tilde{\wa}^{\delta,h}=\wa^{\delta,h} + \varphi^\delta(\cdot,\ra/\e) \na+\eh \big(\bs z^\delta_1(\cdot,\ra/\e)\taub^1+\bs z^\delta_2(\cdot,r/\e)\taub^2\big), \\
&& p^{\delta}(x, y) = t(x)\left( \frac{|A_{\delta}(x) n(x)|^2}{2}\ I + A_{\delta}^2(x) \right)n(x) +2 \int_0^{t(x)} g_\alpha^{\delta}(\pi(x), s, y)\ ds\ \tau^\alpha(x) \\ & &\hspace{+5ex}+\int_0^{t(x)} g_3^{\delta}(\pi(x), s, y)\ ds\ \tau^3(x).
\end{eqnarray*}
where $\va^{\delta,h}$ is defined in (\ref{eq:16}) and $\wa^{\delta,h}$ is defined by the property (\ref{eq:jelgotovo}).

Notice that, similarly as before:
\begin{align}
\e\left( \nabla\t w^{\delta, h} - \nabla w^{\delta, h} \right) &  \mbox{ bounded in }L^{\infty}(S)
\\
\e\left( \Omega_{w^{\delta,h}}-\Omega_{\tilde{w}^{\delta,h}} \right) & \mbox{ bounded in }L^{\infty}(S)
\\
\e\left(b_{\tilde{w}^{\delta,h}}-\left(b_{w}^{\delta,h}+ \tfrac{1}{\e^2}\sum_{i,j=1}^2 \partial_{y_i y_j} \varphi^\delta \taub^i \otimes \taub^j\right) \right) &  \mbox{ bounded in }L^{\infty}(S) \\
\|q_{\tilde{\wa}^{\delta,h}}\|_{L^2}&  \mbox{ bounded independately of } \delta,h \\
\label{citat4}\tfrac{1}{\e} \left(q_{\tilde{\wa}^{\delta,h}} - \left(q_{\wa^{\delta,h}}+\dot{\B}^\delta\right) \right)&  \mbox{ bounded in }L^{\infty}(S).
\end{align}

Now we continue as in the proof of Proposition \ref{prop:peter111}, after concluding that $\tilde{\wa}^{\delta,h}$ satisfies the condition (\ref{eq:15}).
Boundedness of the right hand side of (\ref{eq:iggg2122}) follows easily.
 It can be easily seen that (\ref{eq:15}), (\ref{citat2}) are valid. Instead of (\ref{irre-4}) we have (\ref{citat4}).
\end{proof}

\appendix

\section{Auxiliary results}
In the sequel we consider the sequence $\eh \to 0$ as $h \to 0$ and $\Omega \subset \R^3$ a Lipschitz domain and $\taub_i=\ee_i$,  for $i=1,2,3$, where $\ee_i$ are standard coordinate vectors.
The set $\mathcal{Y}$ can be considered as the set $[0,1)^3$ i.e. $[0,1)^2$ with the topology of torrus.
The claims can be trivially extended to $\R^n$. For the proofs see e.g. \cite{Allaire-92}. For the last claim see \cite[Lemma~3]{Velcic-12}.
\begin{lemma}\label{L:two-scale}
  \begin{enumerate}[(i)]
  \item Any sequence that is bounded in $L^2(\Omega)$ admits a two-scale
    convergence subsequence.
  \item Let $f \in L^2(\Omega \times \calY)$ and $\{f^h\}_{h>0} \subset L^2(\Omega)$ be such that $f^h\wtto f(x,y)$
    weakly. Then $f^h\wto\int_Y f(\cdot,y)\,dy$ weakly in $L^2(\Omega)$.
  \item Let $f^0 \in L^2(\Omega)$ and $\{f^h\}_{h>0} \subset L^2(\Omega)$ be such that $f^h\wto
    f^0$ weakly in
    $L^2$. Then (after passing to subsequences) we have $f^h\wtto f^0(x)+\t f(x,y)$ for some $\t f\in L^2(\Omega\times
    Y)$ with $\int_Y\t f(\cdot,y)\,dy=0$ almost everywhere in $\Omega$.
  \item Let $f^0 \in L^2(\Omega)$ and $\{f^h\}_{h>0} \subset L^2(\Omega)$ be such that $f^h\to f^0$ strongly in
    $L^2$. Then $f^h\stto f^0(x)$.
  \item Let $f^0 \in H^1(\Omega)$ and $\{f^h\}_{h>0} \subset H^1(\Omega)$ be such that $f^h\wto f^0$ weakly in $H^1$.
    Then (after passing to subsequences)
    \begin{equation*}
      \nabla f^h\wtto \nabla f^0+\nabla_y\phi(x,y)
    \end{equation*}
    for some $\phi\in L^2(\Omega, \dot{H}^1(\mathcal Y))$.
  \item Let $f^0 \in H^2(\Omega)$ and $\{f^h\}_{h>0} \subset H^2(\Omega)$ be such that $f^h\wto f^0$ weakly in $H^2$.
    Then (after passing to subsequences)
    \begin{equation*}
      \nabla^2 f^h\wtto \nabla^2 f^0+\nabla^2_y\phi(x,y)
    \end{equation*}
    for some $\phi\in L^2(\Omega, \dot{H}^2(\mathcal Y))$.
  \end{enumerate}
\end{lemma}

For brevity we shall write $\e$ instead of $\eh$.
At several places in our argument we are only interested in the
oscillatory part of the two-scale limit.
In the following, we introduce as in \cite{Horneuvel12} the special
notation $\otto$ for that purpose. As a motivation consider a
sequence $\{f^h\}_{h>0} \subset L^2(\Omega)$ with weak two-scale limit $f \in
L^2(\Omega\times Y)$. Consider
\begin{equation*}
  f^0(x):=\int_Yf(x,y)\,dy\qquad\text{and}\qquad \t f(x,y)=f(x,y)-f^0(x).
\end{equation*}
According to Lemma~\ref{L:two-scale} the function $f^0$ is the weak limit of $f^h$.
We call $\t f(x,y)$ the oscillatory part of $f$. Evidently we have
\begin{multline}\label{eq:ren}
  \lim\limits_{h\to 0}\int_\Omega
  f^h(x)\varphi(x)g(\tfrac{x'}{\e})\,dx=\iint_{\Omega\times Y} \t f(x,y)\varphi(x)g(y)\,dy\,dx\\
  \text{for all $\varphi\in C^\infty_0(\Omega)$ and $g\in C^\infty(\mathcal
    Y)$ with $\int_Yg\,dy=0$.}
\end{multline}
Motivated by that we  introduce the following vocabulary (see \cite{Horneuvel12}):
\begin{definition}
  For a sequence $(f^h)_{h>0}\subset L^2(\Omega)$ and $\t f\in
  L^2(\Omega\times Y)$ with $\int_Y\t f(\cdot,y)\,dy=0$ almost everywhere in $\Omega$ we write
  $$
  f^h \otto \t f(x,y),
  $$
  if \eqref{eq:ren} holds for all  $\varphi\in C^\infty_0(\Omega)$ and $g\in C^\infty(\mathcal
    Y)$ with $\int_Yg\,dy=0$.
\end{definition}

\begin{lemma}
  \label{L:osc}
  Let $f^0$ and $f^h\in L^2(\Omega)$ be such that $f^h\wto f^0$ weakly in
  $L^2(\Omega)$ and $f^h\otto \t f(x,y)$. Then $f^h\wtto f^0(x)+\t
  f(x,y)$ weakly two-scale.
\end{lemma}
\begin{proof}
  Straightforward.
\end{proof}

 The
following Lemma was needed in the proof of Proposition \ref{tm:1}.
\begin{lemma}\label{L:two-scale-h1}
  \begin{enumerate}[(i)]
 \item  Let $f^0$ and $f^h\in H^1(\Omega)$ be such that $f^h\wto f^0$ weakly in
  $H^1(\Omega)$ and assume that
  \begin{equation*}
    \nabla f^h\wtto \nabla f^0 + \nabla_y\phi(x,y)
  \end{equation*}
  for some $\phi\in L^2(\Omega; \dot{H}^1(\mathcal Y))$. Then
  \begin{equation*}
    \frac{f^h}{\e}\otto \phi.
  \end{equation*}

\item Let $f^0$ and $f^h\in H^1(\Omega)$ be such that $f^h\wto f^0$ weakly in
  $H^2(\Omega)$ and assume that
  \begin{equation*}
    \nabla^2 f^h\wtto \nabla^2 f^0 + \nabla^2_y\phi(x,y)
  \end{equation*}
  for some $\phi\in L^2(\Omega;\dot{H}^2(\mathcal Y))$. Then
  \begin{equation*}
    \frac{f^h}{\e^2}\otto \phi.
  \end{equation*}
\end{enumerate}

\end{lemma}
\begin{proof}

The proof of (i) is given in \cite[Lemma~3.7]{Horneuvel12}.
Here we prove (ii) which goes in an analogous way.  Let $G$ denote the unique solution in $C^\infty (\mathcal
  Y)$ to
  \begin{equation*}
    -\Delta_y G=g,\qquad \int_Y G\,dy=0.
  \end{equation*}
We put $G^h(x)=g(\tfrac{x'}{\e})$. Then
$\Delta_y G^h(x)=\frac{1}{\e^2} g(\tfrac{x}{\e}) $ and $\nabla G^h(x)=\tfrac{1}{\e} \nabla_y G(\tfrac{x}{\e})$.

\begin{eqnarray}
     \frac{1}{\e^2}\int_\Omega f^h(x)g(\frac{x}{\e^2})\psi(x)\,dx &=&
    \int_S f^h \Delta G^h \psi\, dx\\
    \nonumber
    &=&-\int_\Omega  \nabla f^h\cdot  \nabla (G^h\psi)\,dx
    -2\int_\Omega f^h( \nabla G^h\cdot \nabla\psi)\,dx
    \\ \nonumber & & -\int_\Omega f^h\,G^h  \Delta \psi\,dx \\
    \nonumber &=& \int_\Omega \Delta  f^h   (G^h\psi) \, dx
     -2\int_\Omega f^h( \nabla G^h\cdot \nabla\psi)\,dx
    \\ \nonumber & & -\int_\Omega f^h\,G^h  \Delta \psi\,dx \\ \nonumber
    &\to &\iint_{\Omega \times Y} \Delta_y \phi(x,y) G(y) \psi  \,dx\,dy\\ \nonumber
    &=& \iint_{\Omega \times Y}  \phi(x,y)  g(y) \psi  \,dx\,dy,
\end{eqnarray}
where we have used (i) for the claim $\tfrac{f^h}{\e} \otto 0$ that is used to conclude $\int_\Omega f^h( \nabla G^h\cdot \nabla\psi)\,dx\to 0$ .

\end{proof}

The following proposition can be found in \cite[Proposition~2]{FJM-06}
\begin{proposition}\label{P:FJM}
Let $\Omega$ be a bounded Lipschitz domain
in $\R^n$ and let $1 < p < \infty$, $k\in \mathbf{N}$ and $\lambda>0$. Suppose that $u\in W^{k,p}(\Omega)$ and
let
$$|u|_k(x) :=\sum_{|\alpha|\leq k} |\nabla^\alpha u(x)|.$$
Then there exists $u^\lambda \in W^{k,\infty}$ such that
\begin{eqnarray*}
\| u^\lambda\|_{W^{k,\infty}} &\leq& C(p,k,\Omega)\lambda, \\
\left| \{ x \in \Omega: \ u^\lambda(x) \neq u(x) \}\right| & \leq & \frac{C(p,k)}{\lambda^p} \int_{|u|_k\geq \lambda/2} |u|_k^p \ud x, \\
\| u^\lambda \|_{W^{k,p}} \leq C(p,k,\Omega)\|u\|_{W^{k,p}}.
\end{eqnarray*}
In particular
$$ \lim_{\lambda \to \infty} \lambda^p \left| \{ x\in \Omega: \ u^\lambda(x) \neq u(x) \}\right|=0,$$
and
$$ \lim_{\lambda \to \infty} \| u^\lambda-u\|_{W^{k,p}}=0. $$

\end{proposition}

The following diagonalization lemma is due to \cite[Corollary 1.16]{Attouch-84}:
\begin{lemma}
  \label{L:attouche}
  Let $g:[0,\infty)\times[0,\infty)\to [0,\infty)$ and suppose that
  \begin{equation*}
    \limsup\limits_{\delta\to 0}    \limsup\limits_{h\to 0}g(\delta,h)=0.
  \end{equation*}
  Then there is a monotone function
  $(0,\infty)\ni h\mapsto\delta(h)\in(0,\infty)$ with $\lim_{h\to 0}\delta(h)=0$ and $\limsup_{h\to0}g(\delta(h),h)=0$.
\end{lemma}

{\bf Acknowledgement.} Both authors were supported by Deutsche Forschungsgemeinschaft grant no. HO-4697/1-1. Part
of this work was completed while the second author was affiliated with BCAM.
\\

\bibliographystyle{alpha}

\end{document}